\documentclass[11pt,oneside]{amsart}
\usepackage{graphicx} 
\usepackage{float}
\usepackage{bbm}
\usepackage[margin=3cm]{geometry}
\usepackage{color}
\usepackage{amsfonts}
\usepackage{algorithm}
\usepackage{algorithmic}
\usepackage{latexsym, amssymb, amsmath, amscd, amsthm, amsxtra}
\usepackage{mathtools}
\usepackage{enumerate}
\usepackage[all]{xy}
\usepackage{mathrsfs}
\usepackage{fancyhdr}
\usepackage{listings}
\usepackage{hyperref}
\usepackage{enumitem}
\usepackage{comment}
\usepackage{tikz}
\newcommand{\slitN}{S^{N,\lambda}}

\def\P{\mathbb{P}}
\def\E{\mathbb{E}}
\def\chl{\mathcal{C}}

\def\im{\operatorname{Im}}
\def\rea{\operatorname{Re}}

\newtheorem{theorem}{Theorem}[section]

\newtheorem{proposition}[theorem]{Proposition}

\newtheorem{lemma}[theorem]{Lemma}
\newtheorem{corollary}[theorem]{Corollary}
\newtheorem{definition}[theorem]{Definition}

\newtheorem{remark}[theorem]{Remark}

\usetikzlibrary{patterns,arrows.meta,positioning,calc}

\numberwithin{equation}{section}

\title{One-arm domination time in Cylindrical Hastings-Levitov$(0)$}

\author{Guanyi Chen$^1$}
	\address[Guanyi Chen]{School of Mathematical Sciences, Peking University, Beijing, China}
	\email{cgy7698@stu.pku.edu.cn}
	\thanks{$^1$School of Mathematical Sciences, Peking University, Beijing, China}

	\author{Eviatar B. Procaccia$^2$}
	\address[Eviatar B. Procaccia]{Faculty of Data and Decision Sciences, Technion - Israel Institute of Technology, Haifa, Israel}
	\email{eviatarp@technion.ac.il}
	\thanks{$^2$Faculty of Data and Decision Sciences, Technion - Israel Institute of Technology, Haifa, Israel}

	\author{Yuxuan Zong$^1$}
	\address[Yuxuan Zong]{School of Mathematical Sciences, Peking University, Beijing, China}
    
	\email{yxzong25@stu.pku.edu.cn}

\begin{document}
\begin{abstract}
The cylindrical Hastings-Levitov$(0)$ admits a single infinite connected tree (arm). For a cylinder of width $N$ and particles of size $\lambda$, {we consider the first time $\upsilon_{N, \lambda}$ after which only the unique infinite tree receives particles}. We prove that $\frac{cN^2}{\lambda^3} \le \E[\upsilon_{N, \lambda}]\le\frac{CN^2}{\lambda^3}$, and establish an exponential tail for $\upsilon_{N, \lambda}$. Moreover, we obtain an asymptotic bound to the expected total number of trees, and the last time a new tree emerges. 
\end{abstract} 
\maketitle
\tableofcontents
\section{Introduction}
Diffusion Limited Aggregation (DLA) is a set-valued stochastic process defined on a graph, where particles sequentially attach according to the harmonic measure, or the limiting distribution from infinity of the first hitting location of a set. Hastings and Levitov \cite{hastings1998laplacian} constructed an off-lattice version of {this aggregation mechanism} via a composition of conformal maps on the exterior of the unit disk $\mathbb{D}$. Consider the conformal map $\phi^{\Delta}:\mathbb{C}\setminus \mathbb{D}\mapsto\mathbb{C}\setminus (\mathbb{D}\cup [1,1+{\Delta}])$, normalized so that $\phi^{\Delta}(z)=cz+a+\frac{b}{z}+\cdots,~c>0$ as $|z|\rightarrow\infty$. For $\theta_k\sim\text{Uniform}(\partial \mathbb{D})$ i.i.d.\ and a sequence of particle sizes ${\Delta}_k$, define $\phi_k(z)=e^{i\theta_k}\phi^{{\Delta}_k}(e^{-i\theta_k}z)$. Hastings-Levitov($\alpha$) is the process given by $\Phi_n=\phi_1\circ\phi_2\circ\cdots\circ\phi_n$ with the choice of particle size ${\Delta}_{n}={\Delta}|\Phi'_{n-1}(\theta_n)|^{-\frac{\alpha}{2}}$. For an off-lattice version of DLA, it is common to choose $\alpha=2$, to obtain particles of approximately fixed size. {Since the accumulated conformal map locally rescales inserted particles, it normalizes each particle's size parameter according to the modulus of its derivative at the attachment point to avoid particle-size blow-up.} 

Recently, connections were shown between DLA {initiated} on a long line to the Stationary DLA \cite{procaccia2019stationary,procaccia2020stationary,mu2022scaling} and Stationary Hastings-Levitov \cite{berger2022growth,berger2025logarithmic,procaccia2021dimension}. An {important} feature of the Stationary Hastings-Levitov setting is that no particle size normalization is required to avoid {blow-up}, allowing the aggregate to be constructed using i.i.d. copies of the conformal slit map. 

DLA on a cylinder was widely studied in both the physics and mathematics literature \cite{benjamini2008diffusion,kol1998solution,marchetti2012stationary}. Two main questions that arise in {the} majority of DLA studies are the asymptotic growth rate, and the number of arms. {Most of} these questions remain resistant to rigorous treatment (see \cite{procaccia2020stabilization} for discussion and some rigorous results in a wedge). For DLA in the cylinder $\mathbb{Z}/N \mathbb{Z}\times \mathbb{Z}$, it is easy to prove that only a single infinite arm exists, but it is hard to evaluate the asymptotic growth rate or the one-arm domination time, as a function of $N\to\infty$. Indeed, at any given time $t>0$, {with  probability  $\left(\frac{1}{N}\right)^N$},  the next $N$ particles will connect into a horizontal line of {one unit above the current} maximal height at time $t$ (see Figure \ref{fig:dla_one_arm}).

\tikzset{every picture/.style={line width=0.75pt}} 

\begin{figure}\center
	\begin{tikzpicture}[yscale=0.5,xscale=0.5]
		\draw[step=1cm,blue,thin, dotted] (0,0) grid (5,6);
            \node at (1,0) [circle,fill=black,inner sep=1pt]{}; \draw [black,very thick] (1,0) -- (1,1);
            \node at (1,1) [circle,fill=black,inner sep=1pt]{};
            \draw [black,very thick] (1,1) -- (1,2);
            \node at (1,2) [circle,fill=black,inner sep=1pt]{};
            \draw [black,very thick] (1,2) -- (1,3);
            \node at (1,3) [circle,fill=black,inner sep=1pt]{};
            \draw [black,very thick] (1,2) -- (2,2);
            \node at (2,2) [circle,fill=black,inner sep=1pt]{};
                \draw [black,very thick] (2,2) -- (2,3);
            \node at (2,3) [circle,fill=black,inner sep=1pt]{};
            
            \node at (2,0) [circle,fill=black,inner sep=1pt]{};
             \draw [black,very thick] (2,0) -- (2,1);
             \node at (2,1) [circle,fill=black,inner sep=1pt]{};
              \draw [black,very thick] (2,1) -- (3,1);
             \node at (3,1) [circle,fill=black,inner sep=1pt]{};
             
            \node at (4,0) [circle,fill=black,inner sep=1pt]{};
            \node at (4,1) [circle,fill=black,inner sep=1pt]{};
            \draw [black,very thick] (4,0) -- (4,1);
            \node at (4,1) [circle,fill=black,inner sep=1pt]{};
            \draw [black,very thick] (4,1) -- (4,2);
            \node at (4,2) [circle,fill=black,inner sep=1pt]{};
            \draw [black,very thick] (4,2) -- (5,2);
            \node at (5,2) [circle,fill=black,inner sep=1pt]{};
            \draw [black,very thick] (4,2) -- (4,3);
            \node at (4,3) [circle,fill=black,inner sep=1pt]{};
            \draw [black,very thick] (4,3) -- (4,4);
            \node at (4,4) [circle,fill=black,inner sep=1pt]{};
            \draw [black,very thick] (4,4) -- (5,4);
            \node at (5,4) [circle,fill=black,inner sep=1pt]{};
            \draw [black,very thick] (0,4) -- (1,4);
            \node at (1,4) [circle,fill=black,inner sep=1pt]{};
            \draw [black,very thick] (1,4) -- (2,4);
            \node at (2,4) [circle,fill=black,inner sep=1pt]{};
            \draw [black,very thick] (2,4) -- (3,4);
            \node at (3,4) [circle,fill=black,inner sep=1pt]{};
            
	\end{tikzpicture}
	\caption{ One-arm domination event for DLA in a cylinder.} \label{fig:dla_one_arm}
\end{figure}

In an off-lattice version called Cylindrical Hastings-Levitov$(0)$ (CHL), the growth rate was established in \cite{procaccia2023cylindrical}: For CHL with particles of size $\lambda$, on a cylinder of width $N$, after $tN$ particles, on average the height is $\frac{\pi}{2}\lambda^2 t$. It was established in \cite{norris2012hastings} that only a single tree persists to grow ad infinitum. {We call the first time after which only the unique infinite tree receives particles
the one-arm domination time, and denote it $\upsilon_{N, \lambda}$.} (see Figure \ref{fig:chlsimulation} for computer simulations of CHL and DLA on a cylinder, {shown at times that are, with high probability, after the one-arm domination time}). 
\begin{figure}[H]
    \centering
    \includegraphics[width=0.3\linewidth]{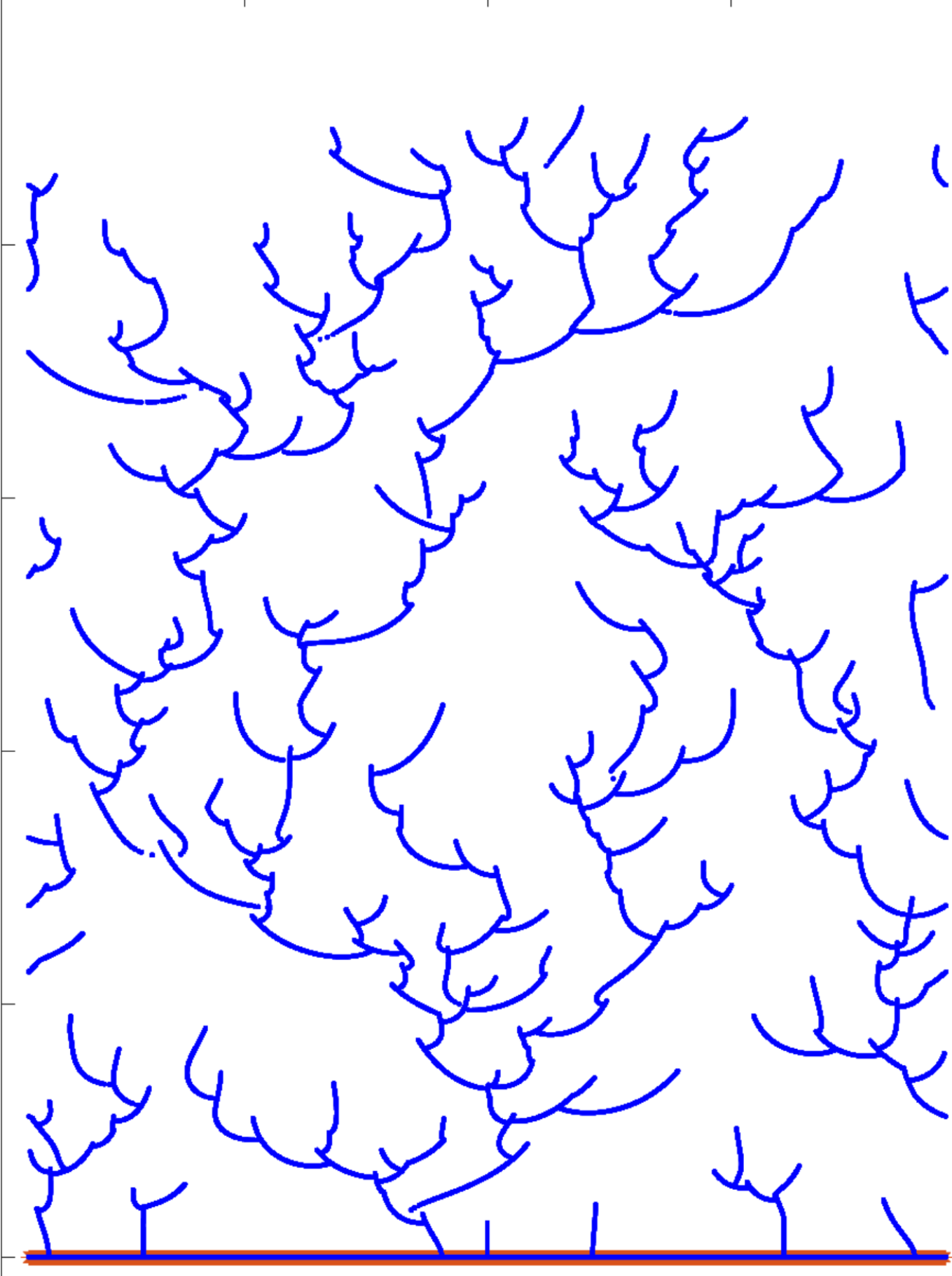}\hskip 1cm
    \includegraphics[width=0.3\linewidth]{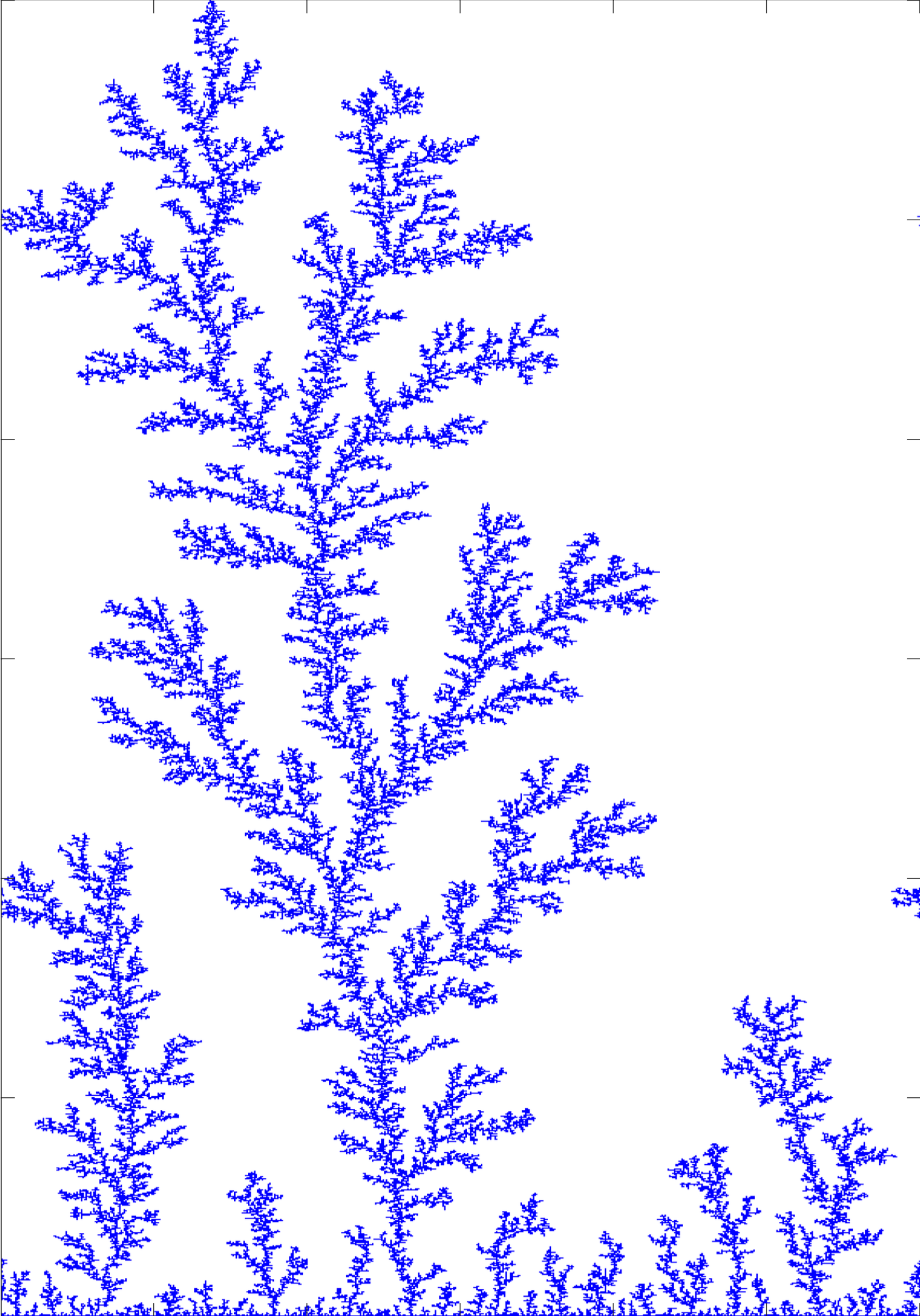}
    \caption{{Illustrative computer simulations:} CHL (left) and DLA on the cylinder (right).}
    \label{fig:chlsimulation}
\end{figure}
Note that though no rigorous connection has been established yet between the cylindrical DLA and CHL, the CHL in itself is an interesting and rich mathematical model, that resembles natural growth processes.

In this paper we study the one-arm domination time $\upsilon_{N, \lambda}$ of CHL. {
We  prove that $\E[\upsilon_{N, \lambda}]\asymp\frac{N^2}{\lambda^3}$, and give an exponential tail for $\upsilon_{N, \lambda}$. The coalescing Brownian flow limit of ~\cite{norris2012hastings} is an important guide for the relevant coalescence scale, but their convergence result is formulated for the harmonic measure flow in a weak-flow topology, whereas the one-arm domination time is not a continuous functional of that topology. Therefore, we develop some techniques that allow us to describe the current state of CHL by intervals on the boundary of a circle, and track how the intervals evolve through time to deal with the observable $\upsilon_{N,\lambda}$. This also allows us to obtain tail bounds for $\upsilon_{N, \lambda}$.
}

\subsection{Definitions}\label{SS:Definitions}
We recall the definition of CHL appearing in \cite{procaccia2023cylindrical}. Define the circle
\begin{align*}
	\mathbb T_N = \mathbb{R} / x \sim x + 2\pi N=[0, 2\pi N), 
\end{align*}
the cylinder
\begin{align}\label{cylinder}
    \mathbb T^N = \{z \in \mathbb{C} : \im(z) > 0\} / z \sim z + 2\pi N, 
\end{align}
the complement of the unit disk
$
\mathbb D_0=\{z\in \mathbb C: |z|>1\},
$
and the upper {half-plane} $\mathbb{H} = \{z \in \mathbb{C} : \im(z) > 0\}$. For two complex numbers $z, w \in \mathbb{C}$, denote $[z, w] = \{tz + (1 - t)w : t \in [0, 1]\}$. 

{Next, suppose $\lambda > 0$ is given, which is the size of one particle in our setting. Then, we construct the cylindrical slit map ``centered" at $0$, which maps $\mathbb T^N$ to $\mathbb T^N \setminus [0, i\lambda]$ and maps $0$ to $i\lambda$. We take }
{
\begin{align}\label{eq-def-delta}
\delta =\delta(N, \lambda) := 1-\frac{2}{e^{\lambda/N}+1}\overset{\text{for large}\ N}{=}\frac{\lambda}{2N}+O\left(\frac{\lambda}{N}\right)^3,
\end{align}
then, we define the cylindrical slit map centered at $0$ by
$$
S^{N,\lambda}(z) = f_N^{-1} \circ g^{-1} \circ \phi_{\delta(N,\lambda)} \circ g \circ f_N(z)\,, 
$$
where
\[
\begin{aligned}
f_N:\mathbb T^N \to \mathbb D_0
&\quad f_N(z)=e^{-iz/N},
&\quad f_N^{-1}:\mathbb D_0 \to \mathbb T^N
&\quad f_N^{-1}(z)=iN\log z,\\
g:\mathbb D_0 \to \mathbb H
&\quad g(z)=i\frac{z-1}{z+1},
&\quad g^{-1}:\mathbb H \to \mathbb D_0
&\quad g^{-1}(z)=\frac{i+z}{i-z},\\
\phi_a:\mathbb H \to \mathbb H\setminus[0,ia]
&\quad \phi_a(z)=\sqrt{z^2(1-a^2)-a^2},&\mbox{ where }a>0.
&
\end{aligned}
\]}
\begin{figure}[H]
\begin{center}
	
	\definecolor{myred}{RGB}{255, 0, 0}
	\definecolor{myblue}{RGB}{0, 0, 255}
	
	\tikzset{
		redcirc/.style={circle, draw=black, fill=myred, thin, inner sep=0pt, minimum size=3mm},
		bluecirc/.style={circle, draw=black, fill=myblue, thin, inner sep=0pt, minimum size=5mm},
	}
	
	\begin{tikzpicture}[scale=0.75]
		
		
		\node (v1) at (-1,-0.3) {\tiny{$-N\pi$}};
		\node (v2) at (1,-0.3) {\tiny{$N\pi$}};
		\draw [ultra thick] (-1,0) to (1,0);
		\draw[pattern=north west lines, pattern color=blue, opacity=0.2] (-1,0) rectangle (1,3);
		\node (v4) at (0,0) [redcirc] {\tiny{$0$}};
		
		\node (v3) at (1.5,1.5) {\tiny{$\overset{f_N}{\to}$}};
		
		\draw[pattern=north west lines, pattern color=blue, opacity=0.2] (1.8,0) rectangle (4.2,3);
		\filldraw[color=black!160, fill=white!5, ultra thick](3,1.5) circle (1);
		\node (v6) at (4,1.5) [redcirc] {\tiny{$1$}};
		
		\node (v5) at (4.6,1.5) {\tiny{$\overset{g}{\to}$}};
		
		\draw[pattern=north west lines, pattern color=blue, opacity=0.2] (5,0) rectangle (8,2);
		\draw [ultra thick] (5,0) to (8,0);
		\node (v7) at (6.5,0) [redcirc] {\tiny{$0$}};
		
		\node (v8) at (8.3,1.5) {\tiny{$\overset{\phi_\delta}{\to}$}};
		
		\draw[pattern=north west lines, pattern color=blue, opacity=0.2] (8.7,0) rectangle (11.7,2);
		\draw [ultra thick] (8.7,0) to (11.7,0);
		\draw [ultra thick] (10.2,0) to (10.2,0.7);
		\node (v8) at (10.2,0.7) [redcirc] {\tiny{$i\delta$}};
		
		
		\node (v8) at (10.2,-0.5) {\tiny{${\downarrow}{g^{-1}}$}};
		
		\draw[pattern=north west lines, pattern color=blue, opacity=0.2] (8.7,-4) rectangle (11.7,-1);
		\filldraw[color=black!160, fill=white!5, ultra thick](10,-2.5) circle (1);
		\draw [ultra thick] (11,-2.5) to (11.4,-2.5);
		\node (v9) at (11.4,-2.5) [redcirc] {};
		\node (v10) at (11.4,-2.1) {\tiny{$\frac{1+\delta}{1-\delta}$}};
		
		\node (v8) at (7.9,-2.5) {\tiny{$\overset{f_N^{-1}}{\leftarrow}$}};
		
		\begin{scope}[shift={(6,-4)}]
			\node (v11) at (-1,-0.3) {\tiny{$-N\pi$}};
			\node (v12) at (1,-0.3) {\tiny{$N\pi$}};
			\draw [ultra thick] (-1,0) to (1,0);
			\draw[pattern=north west lines, pattern color=blue, opacity=0.2] (-1,0) rectangle (1,3);
			\draw [ultra thick] (0,0) to (0,1.3);
			\node (v14) at (0,1.3) [redcirc] {\tiny{$i \lambda$}};
		\end{scope}
		
		\draw [-stealth](0,-0.5) -- (4.5,-3);
		\node (v15) at (2.7,-1.5) {$\slitN$};
	\end{tikzpicture}
\end{center}
	\caption{The cylinder slit map $S^{N,\lambda}(z):\mathbb T^N\to\mathbb T^N\setminus[0,i\lambda]$}
	\label{basic_slit_map}
\end{figure}
To {verify that} $S^{N,\lambda}(0) = i\lambda$, {it suffices to note that}
\begin{equation}\label{map}
	{0 \xrightarrow{f_N} 1 \xrightarrow{g} 0 \xrightarrow{\phi_\delta} i\delta \xrightarrow{g^{-1}} \frac{1 + \delta}{1 - \delta} \xrightarrow{f_N^{-1}} iN \log\left(\frac{1 + \delta}{1 - \delta}\right) = i\lambda\,.}
\end{equation}
{Similarly, we can verify that $S^{N,\lambda}(\mathbb T_N) = \mathbb T_N \cup [0, i\lambda]$.}

{Next, given the above definition of a slit map centered at $0$, we define a slit map centered at general position.} The slit map $S^{N,\lambda}_x$ {centered at $x \in \mathbb T_N$ with particle size $\lambda$ is defined by} 
\begin{equation}
S^{N,\lambda}_x(z) = \rea(z) + S^{N,\lambda}(z - x) - (\rea(z - x) \mod 2\pi N),\quad z\in\mathbb{T}^N.
\label{eq:identity}
\end{equation}
{In particular,} $S^{N,\lambda}_0(z)=S^{N,\lambda}(z)$. 

{Given the general definition of slit maps, we define the CHL process on $\mathbb T^N$ as follows. }
 
\begin{definition}[$\operatorname{CHL}_N$]
Consider a Poisson point process $P$ with intensity 1 on $\mathbb R_+\times \mathbb T_N $. Let $I_{t,N}$ be the set of points in $P$ distributed on $[0, t]\times \mathbb T_N $. Almost surely, $P$ has finitely many points in any compact set. Hence we can write 
\[
    I_{t, N} = \{(t_1, x_1), (t_2, x_2), \ldots, (t_n, x_n)\}
\]
such that $0 < t_1 < t_2 < \cdots < t_n \leq t$ and $\forall i, 0 < x_i \leq 2\pi N$. $\operatorname{CHL}_N$ is the c\'adl\'ag function $\chl^{N,\lambda}_t(z)$ such that
$$
{\chl^{N,\lambda}_t(z) =S^{N,\lambda}_{x_1} \circ S^{N,\lambda}_{x_2} \circ \cdots \circ S^{N,\lambda}_{x_n}(z)}
$$
and we denote $\{\mathcal C^{N,\lambda}_t(\cdot)\}_{t \geq 0} \lhd P$ in this case. The backward $\operatorname{CHL}_N$ process is defined by
\begin{equation}
{\tilde{\chl}^{N,\lambda}_t(z) =S^{N,\lambda}_{x_n} \circ S^{N,\lambda}_{x_{n-1}} \circ \cdots \circ S^{N,\lambda}_{x_1}(z).}
\label{eq:backwardchln}
\end{equation}
\end{definition}
Note that the {backward} process is not Markovian. However, for any fixed time, it is equidistributed as the original process (see Section 2 in \cite{berger2022growth}) and admits more tractable analysis.  

\begin{definition}[Trees]\label{def:tree}
    We call the connected components of 
    $$\chl^{N,\lambda}_t( \mathbb T_N \times \{0\})\setminus (\mathbb T_N \times \{0\})$$
    the trees of CHL$_N$, and denote the number of trees at time $t$ as $\mathcal{N}_{t}$. 
    By \cite{norris2012hastings} there is a single tree that persists in growing ad infinitum { (we include a self-contained proof in Proposition \ref{geometric properties of CHL})}.
    
    {Define the one-arm domination time $\upsilon_{N, \lambda}$ as the minimal time $\upsilon$ such that the infinite tree receives a particle at time $\upsilon$, and no tree other than the infinite tree receives a particle after $\upsilon$. Also, define the tree completion time $\omega_{N, \lambda}$ as the last time a new tree grows on $\mathbb T_N$.}
\end{definition}

\subsection{Main results} 
{In this section we state the main results. The first main result of this paper establishes that the expectation of the one-arm domination time is of order $\lambda^{-3}N^2\,$.}
\begin{theorem}\label{thm: one-arm domination time}
    {There exist universal constants $c,C>0$ such that for every fixed $\lambda>0$, there exists $N_0(\lambda)<\infty$ such that for all $N\ge N_0(\lambda)$, the one-arm domination time $\upsilon_{N,\lambda}$ in $\mathrm{CHL}_N$ satisfies} 
    \[
        \frac{cN^2}{\lambda^3}\le \mathbb{E}[\upsilon_{N, \lambda}]\le \frac{CN^2}{\lambda^3}.
    \]
\end{theorem}
Furthermore, we obtain an exponential tail for the one-arm domination time.
\begin{theorem}\label{thm:tail of one-arm domination time}
    {There exist universal constants $C_1,C_2,C_3,C_4>0$ such that for every $\lambda>0$, there exists $N_0(\lambda)<\infty$ such that for all $N\ge N_0(\lambda)$ and all $t>0$,} 
     \begin{align*}
         C_3\text{e}^{-C_4t}\le \mathbb{P}\left(\upsilon_{N, \lambda}>\lambda^{-3}N^2t\right)\le C_1\text{e}^{-C_2t}.
     \end{align*}
\end{theorem}

For the tree completion time $\omega_{N, \lambda}$, we also obtain an upper bound to its expectation.
\begin{theorem}\label{thm: tree completion time}
    {For every fixed $\lambda>0$, $\varepsilon>0$ and  all sufficiently large $N$,} 
    \begin{align*}
        {\mathbb{E}[\omega_{N, \lambda}]\le (1+\varepsilon)\frac{\log N}{2\lambda}.}
    \end{align*}
\end{theorem}
\begin{remark}
    We can directly obtain the following lower bound result. There exists $C_0>0$ such that
    $\mathbb{E}[\omega_{N, \lambda}]>C_0$. We conjecture that 
    \begin{align*}
        {\lim_{N\to \infty}\frac{\mathbb{E}[\omega_{N, \lambda}]}{\log N}=\frac{1}{2\lambda}.}
    \end{align*}
\end{remark}

Finally, we compute the asymptotics of the expected number of trees when $N$ tends to infinity. {Recall from Definition \ref{def:tree} that} $\mathcal{N}_{t}$ denotes the number of trees at time $t$ in $\operatorname{CHL}_N$. We can obtain the following result.
\begin{theorem}\label{thm: number of trees}
    There exists $\mathcal{N}_{\infty}$ such that
    \begin{align*}
        \mathcal{N}_{t} \to \mathcal{N}_{\infty} \quad \mathrm{a.s.}
    \end{align*}
    as $t \to \infty$, and 
    \begin{align*}
        \mathbb{E}[\mathcal{N}_{\infty}]= \frac{\pi}{2 \arctan\frac{\delta(N,\lambda)}{\sqrt{1-\delta(N,\lambda)^2}}}
    \end{align*}
    (See \eqref{eq-def-delta} for the definition of $\delta(N,\lambda)$). In particular, 
    \begin{align*}
        \lim_{N\to \infty}\frac{\mathbb{E}[\mathcal{N}_{\infty}] }{N }=\frac{\pi}{\lambda}.
    \end{align*}
\end{theorem}
{A direct corollary gives the expected number of children of a single particle.} 
\begin{corollary}\label{cor:degree of a single particle}
    {Assume that a particle is attached at time $t_0$. Let $D_t$ denote the number of particles directly attached to this particle during $(t_0,t]$. Then there exists $D_{\infty}$ such that}
    \begin{align*}
        {D_{t} \to D_{\infty} \quad \mathrm{a.s.}}
    \end{align*}
    {as $t \to \infty$, and}
    \begin{align*}
        {\mathbb{E}[D_{\infty}]= 1.}
    \end{align*}

\end{corollary}
\subsection{Proof overview}
{Note that it can be shown that 
\begin{align*}
NS^{1,\lambda / N}_{x/N}\left(\frac{\cdot}{N}\right) = S^{N,\lambda}_x(\cdot)\,,\quad N\mathcal C^{1,\lambda / N}_{t}\left(\frac{\cdot}{N}\right) = \mathcal C^{N,\lambda}_t(\cdot)\,,
\end{align*}
and
\begin{align*}
\delta\left(1,\frac{\lambda}{N}\right) = \delta(N,\lambda)\,.
\end{align*}
Therefore, we can scale the cylinder $\mathbb T^N =\mathbb T_N \times \mathbb R_+$ to $\mathbb T^1 =\mathbb T_1 \times \mathbb R_+$, and work with $\bigl(1,\tfrac{\lambda}{N}\bigr)$ to simplify our analysis. Also, for simplicity of notations, in the rest of the article we abbreviate $\delta(N,\lambda) = \delta\bigl(1,\tfrac{\lambda}{N}\bigr)$ by $\delta$ when there is no ambiguity.
}

{To begin with, we consider the properties of the inverse of a slit map,} and investigate how an interval on the unit torus is influenced by performing an inverse slit map at a uniform position. Our investigation is motivated by the following observation: For a subset $\mathcal T$ of $\mathcal C^{1,\lambda/N}_{t_n}(\{ 0\} \times \mathbb T_1)$ in our rescaled CHL ($\mathcal T$ could be a tree or a subinterval of $\{ 0\} \times \mathbb T_1$, for example), the probability that the $(n+1)$-th particle hits $\mathcal T$ equals $$\frac{1}{|\mathbb T_1|}|(\mathcal C^{1,\lambda/N}_{t_{n}})^{-1}(\mathcal T)|,$$ which becomes $$\frac{1}{|\mathbb T_1|}|(S^{1,\lambda/N}_{x_{n+1}})^{-1} \circ (\mathcal C^{1,\lambda/N}_{t_{n}})^{-1}(\mathcal T)| = \frac{1}{|\mathbb T_1|}|(\mathcal C^{1,\lambda/N}_{t_{n+1}})^{-1}(\mathcal T)|$$ after attaching a new particle $x_{n+1}$ at time $t_{n+1}$ in our rescaled CHL. Therefore, our investigation into change of intervals on the torus by a random inverse slit map helps us to derive useful properties of the evolution of $\operatorname{CHL}_N$. {To characterize the hitting probability of all trees on a given interval $I$ at time $t$, it suffices to pull back the interval by the standard CHL map \(C_t\) and get the inverse map of $I$ as an arc interval, which we will denote by $\mathcal (C^{1,\lambda / N}_t)^{-1}(I)$ in the following context. We will call this method the \textit{interval characterization of $\operatorname{CHL}_N$} in the following context.}

In Section \ref{sec:Expectation of one-arm domination time}, we begin by dealing with the lower bound for Theorem~\ref{thm: one-arm domination time} {by the interval characterization of $\operatorname{CHL}_N$}. Intuitively, if there exists a dominating tree $\mathcal T$, i.e. a tree {that} captures most of the harmonic measure, then the probability that $\mathcal T$ is hit by a new random particle, namely $\tfrac{1}{|\mathbb T_1|}|(\mathcal C^{1,\lambda/N}_{t_{n}})^{-1}(\mathcal T)|$, should be very close to 1. {Suppose} we divide $\mathbb T_1$ into two semicircular arc intervals $I$ and $I'$, and the domination of a tree $\mathcal T$ happens after the $n$th particle arrives{. Then,} at least one of $(\mathcal C^{1,\lambda/N}_{t_{n}})^{-1}(I)$ and $(\mathcal C^{1,\lambda/N}_{t_{n}})^{-1}(I')$ should contain $(\mathcal C^{1,\lambda/N}_{t_{n}})^{-1}(\mathcal T)$ {and occupies almost all of} $\mathbb T_1$. Therefore, $\E[\upsilon_{N,\lambda}]$ is bounded below by the expectation of the time when one of $\tfrac{1}{|\mathbb T_1|}|(\mathcal C^{1,\lambda/N}_{t_{n}})^{-1}(I)|$ and $\tfrac{1}{|\mathbb T_1|}|(\mathcal C^{1,\lambda/N}_{t_{n}})^{-1}(I')|$ becomes close to 1, which can be estimated by analogy {with} the hitting time of a simple random walk.

We then proceed to deal with the upper bound for Theorem~\ref{thm: one-arm domination time}, where {we need to apply the interval characterization of $\operatorname{CHL}_N$ in a more sophisticated way. The sketch of our proof is as follows.} When the first particle is attached to $\mathbb T_1$, we choose an interval $I^{(1)}$ of length $|I^{(1)}| = \delta $ on $\mathbb T_1$ such that a unique tree emanates from $I^{(1)}$, {which we will show is always possible.} Then we perform i.i.d.\ random inverse slit maps $(S_x^{1,\lambda/N})^{-1}$ where $x \sim \mathsf{Unif}(\mathbb T_1)$, and update $I^{(1)}$ by $(S_x^{1,\lambda/N})^{-1}(I^{(1)})$ each time. We repeat this process until either $|I^{(1)}| < \delta^3$ or $|I^{(1)}| > 2\pi - \delta^3$ after $T$ mappings, where we will stop this process and rerun the same process for newly constructed $I^{(2)}\,,I^{(3)}\,,\ldots\,.$ {By Lemma~\ref{lem-harmonic-estimations-quadratic}, we have
$\mathbb{E}[T]\asymp \delta^{-2}$.}

Next, by the martingale property of $|I^{(1)}|$, the probability that $|I^{(1)}| > 2\pi - \delta^3$ happens before $|I^{(1)}| < \delta^3$ is $O(\delta)$. If $|I^{(1)}| > 2\pi - \delta^3$, then with probability {$1 - o(1)$}, all new particles later are attached to $I^{(1)}$, which yields a one-arm domination {by our interval characterization of $\operatorname{CHL}_N$}. {Therefore, if we define $\tau$ by the minimal index such that $|I^{(\tau)}| > 2\pi - \delta^3$ happens before $|I^{(\tau)}| < \delta^3$, then the one-arm domination time $\upsilon_{N, \lambda}$ of $\operatorname{CHL}_N$ has an upper bound
\[
T'_N = 1 + \sum_{j=1}^{\tau} T_j\,.
\]}

Since $\tau$ is {stochastically} dominated by the first success time of an independent series of Bernoulli trials with parameter $O(\delta)$, we have $\mathbb{E}[\tau] \asymp \delta^{-1}$. 
{Therefore, we have}
\[
\mathbb{E}[T'_N] \leq 1+\E[T]\E[\tau] \asymp 1+O(\delta^{-1})O(\delta^{-2}) = O(\delta^{-3})\,.
\]

{Also, by the fact that the total arrival rate in the original cylinder is $2\pi N$, the expected time for $O(\delta^{-3})$ particles to arrive is $O(\delta^{-3}/N)=O(N^2/\lambda^3)$ by \eqref{eq-def-delta}.} This completes our analysis of the upper bound for Theorem~\ref{thm: one-arm domination time}.

In Section \ref{sec:Tail of one-arm domination time}, we establish an exponential tail for one-arm domination time by {the} strong Markov property. 
Intuitively, it suffices to show that for sufficiently small $c$ {with high probability,}  $c\delta^{-3}$ {random slit maps} are not enough, and for sufficiently large $C$ {with high probability,}  $C\delta^{-3}$ {random slit maps} are enough. {Again, we will use the interval characterization of $\operatorname{CHL}_N$.}

For the lower bound, the crucial argument is that for any interval of length between $\tfrac{2\pi}{3}$ and $\tfrac{4\pi}{3}$, there is positive probability bounded away from 0 such that its length remains between $\tfrac{2\pi}{3}$ and $\tfrac{4\pi}{3}$ after $c_*\delta^{-3}$ random {inverse slit maps}. Therefore, for trees captured by the semicircular interval $[0 , \pi]$, {their total hitting probability is neither too small nor too large} in the first $C\delta^{-3}$ mappings,  {which rules out one-arm domination during this period}. Repeating such arguments yields an exponentially decaying probability, and in this way we derive our lower bound. For the upper bound, we continue the method in Section~\ref{upper bound}. By Markov's inequality, there is positive probability bounded away from $1$ such that within $C^*\delta^{-3}$ mappings, all marked intervals have length no more than $2\pi-\delta^3$. Using strong Markov property, the probability that all marked intervals have length no more than $2\pi-\delta^3$ within $C^{**}C^*\delta^{-3}$ mappings is bounded exponentially in $C^{**}$, which yields the upper bound. 

In Section \ref{sec:Expectation of tree completion time}, we get an upper bound of the expected tree completion time, {which can be dealt with by investigating the total length of intervals without trees in the interval characterization of $\operatorname{CHL}_N$}. We first derive a recursive formula for the expectation of the total length of non-tree intervals. {Then, since the expectation is summable after $(1+\varepsilon)\frac{\pi}{\lambda}N\log N$ particle arrivals, and the total arrival rate is $2\pi N$, the tree completion time is bounded by $(1+\varepsilon)\frac{\log N}{2\lambda}$.}

In Section \ref{sec:Expected number of trees}, we obtain the asymptotics of the expected number of trees. For fixed $N$, we prove the existence of the {almost surely} limit of the number of trees in the $\operatorname{CHL}_N$ process. Then, we substitute this limit into a recursive  formula for the expectation, and obtain the asymptotic density. As a corollary, we also get the asymptotics of the expected number of children of a single particle  using the same analysis. 

\subsection{Organization of the paper}
In Section \ref{sec:Estimate of Interval variation}, with some properties of the inverse of a slit map, we get the estimate of interval variation. In Section \ref{sec:Expectation of one-arm domination time}, we use the estimate in Section \ref{sec:Estimate of Interval variation} to prove Theorem \ref{thm: one-arm domination time}. In Section \ref{sec:Tail of one-arm domination time}, based on the proof of Theorem \ref{thm: one-arm domination time}, we use the estimate in Section \ref{sec:Estimate of Interval variation} to prove Theorem \ref{thm:tail of one-arm domination time}. In Section \ref{sec:Expectation of tree completion time}, we continue to use the ``marked configurations" to prove Theorem \ref{thm: tree completion time}. In Section \ref{sec:Expected number of trees}, we prove Theorem \ref{thm: number of trees} and use it to prove Corollary \ref{cor:degree of a single particle}. 


\section{Estimate of interval variation}\label{sec:Estimate of Interval variation}
{This section develops the interval estimates used later in this paper. After rescaling the cylinder $\mathbb T^N$ to $\mathbb T^1$, we regard the inverse $\text{CHL}_N$ process as a Markov jump process on boundary intervals driven by the broadened inverse slit map. We establish the martingale property of interval lengths, one-step estimates,  the second moment bound and exit-time bounds.}
\subsection{Conventions and Preliminaries}
We consider the boundary pull-back process associated with the standard $\text{CHL}_N$ as a Markov jump process. Denote $\mathbb T^N =\mathbb T_N \times \mathbb R_+$ for all $N \geq 1$, and denote by $\mathbf{PPP}_{\mathbb T^N}(1, x)$ the Poisson point process with intensity 1 on the time scale and intensity $x$ on the position scale on $\mathbb T^N$. To simplify our analysis, we rescale $\mathbb R_+ \times \mathbb T^N$ to $\mathbb R_+ \times \mathbb T^1$ by $(t, z) \longrightarrow (t, R_N(z))$, where $R_N(z) := z/N$. It can be verified that 
 \begin{equation}\label{rescale}
 R_N^{-1} \circ S^{1,\lambda/N}_{x/N}(\cdot) \circ R_N = S^{N,\lambda}_{x}(\cdot) \mbox{ for all }x \in \mathbb T_N,
 \end{equation}
therefore, there exists a coupling of the law of $\Big(\{\mathcal C^{N,\lambda}_t(\cdot)\}_{t \geq 0}\lhd P\sim \mathbf{PPP}_{\mathbb T^N}(1, 1)\Big)$ and the law of $\Big(\{{\mathcal C}^{1,\lambda/N}_t(\cdot)\}_{t \geq 0}\lhd \widetilde P\sim \mathbf{PPP}_{\mathbb T^1}(1, N)\Big)$ by
\begin{equation*}
R_N^{-1}\circ {\mathcal C}^{1,\lambda/N}_t(\cdot) \circ R_N = \mathcal C^{N,\lambda}_t(\cdot) \,,\quad \widetilde P = R_N(P)
\end{equation*}
such that $\mathbb{P} [(\{\mathcal C^{N,\lambda}_t(\cdot)\}_{t \geq 0}\,, P) = (\{{\mathcal C}^{1,\lambda/N}_t(\cdot)\}_{t \geq 0}\,, \widetilde P)] = 1$. Therefore, it suffices to consider the rescaled $\text{CHL}_N$ : $(\{{\mathcal C}^{1,\lambda/N}_t(\cdot)\}_{t \geq 0})$ driven by the Poisson point process $\tilde{P} \sim \mathsf{id} \times R_N (\mathbf{PPP}_{\mathbb T^N}(1, 1) ) \overset{d}{=} \mathbf{PPP}_{\mathbb T^1}(1, N)$. 

{We write $k$ for the attached particle index, $x_k$ the spatial position of the $k$-th arriving particle, and $t_k$ for the $k$-th Poisson arrival time. Then, the Poisson points driving the rescaled $\text{CHL}_N$ are $\{ (t_j, x_j)\}_{j \geq 1}$ with $\{ t_j\}_{j\geq 1}$ increasing. Also, $\{t_k\}_{k\geq 1}$ is itself a Poisson point process with intensity $2\pi N$ and is independent of  $\{x_k\}_{k\geq 1}$. Hence, for a random index $K \in \sigma(\{x_k:k \geq 1\})$, we have $\mathbb E[t_K | K]=K/(2\pi N)$, and therefore $\mathbb E[t_K]=\mathbb E[K]/(2\pi N)$ whenever $\E[K] < +\infty$.}

{Moreover}, we emphasize the following simplified notation {below}. 
\begin{itemize}
    \item For our rescaled $\text{CHL}_N$, we abbreviate the slit map by $S_x:=S^{1,\lambda/N}_x$ and we abbreviate the rescaled $\text{CHL}_N$ map by $\mathcal C_{t}:=\mathcal C^{1,\lambda/N}_{t}$.
    \item For any $A\subset \mathbb{T}_1$, if $x\in A$, we denote the broadened inverse slit map by
\begin{align}
    S_{x}^{\mathtt{inv}}(A):=
    \begin{cases}
        S_x^{-1}(A \cup[x, x+i\lambda/N])\,, \quad &x \in A\,,\\
        S_x^{-1}(A)\,, &x \not\in A\,
    \end{cases}
\end{align} 
to track the change of probability that $A \cup \{ \text{trees on }A\text{ at time }t\}$ is hit for any $A\subset \mathbb{T}_1$. Also, for $t>0$ we denote
\begin{align}
    \chl_{t}^{\mathtt{inv}}(A):= S_{x_{\#\{ j:t_j \leq t\}}}^{\mathtt{inv}} \circ \ldots \circ S_{x_1}^{\mathtt{inv}} (A )\,.
\end{align}
{Then, by definition of $\chl_{t}^{\mathtt{inv}}(A)$ we have}
\begin{align*}
\chl_{t}^{\mathtt{inv}}(A) = \chl_{t}^{-1}(A \cup \{ \text{trees on }A\text{ at time }t\})\,.
\end{align*}
\end{itemize}

Suppose that there is an interval $I$ with $|I| = a$ where $0 \leq a \leq 2\pi$. Then, by definition we can see that both $\chl_t^{\mathtt{inv}}(I)$ and $\Big|\chl_t^{\mathtt{inv}}(I)\Big|$ {are}  Markov jump processes on the state space $[0,2\pi]$, and the transition probability of the embedded Markov chain of $\Big|\chl_t^{\mathtt{inv}}(I)\Big|$ can be formulated by
\begin{align}
p(|I|,b+db) = \P_{x \sim \mathsf{Unif}(\mathbb T_1)} \big(\big|S_x^{\mathtt{inv}}(I)\big|\in [b,b+db]\big).\label{eq-markov-transition-inverse-slit-map}
\end{align}

{To use this formula}, we first characterize the one-step transition of the embedded Markov chain. For any $I \subset \mathbb T_1$ and $\theta \in \mathbb T_1$ we define $I + \theta =\{ z + \theta : z \in I\}$, which is well-defined. Then by symmetry we have
\begin{align}
p(|I|,b+db) = \P_{y \sim \mathsf{Unif}(\mathbb T_1)} \big(\big|S_0^{\mathtt{inv}}(I+y)\big|\in [b,b+db]\big)\,.\label{eq-markov-transition-inverse-slit-map-2}
\end{align}

{Also, with a slight abuse of notation, for any $x \in \mathbb T_1$ and $\theta \in \mathbb T_1 \setminus \{0\}$ we say $S_x^{\mathtt{inv}}(\theta) = \theta'$ if and only if $S_x^{\mathtt{inv}}(\{\theta\}) = \{\theta'\}$. Recall that $g \circ f_N(z) = \tan\left(z/2N\right)$ in Subsection \ref{SS:Definitions}.} By definition we have
\begin{align}
S_0^{\mathtt{inv}}(\theta) = \theta': \,\theta' = 2\arctan \left(\operatorname{sgn}\bigl(\tan \tfrac{\theta}{2}\bigr)\cdot\sqrt{\frac{\bigl(\tan \tfrac{\theta}{2}\bigr)^2+\delta^2}{1-\delta^2}}\right)\,\label{eq-S0-1,delta-inverse-explicit}
\end{align}
for any $\theta \in \mathbb T_1 \setminus \{0, \pi\}$, and 
\begin{align*}
S_0^{\mathtt{inv}}(\pi) = \pi,\qquad S_0^{\mathtt{inv}}(\{0\}) = \left[-2\arctan\tfrac{\delta}{\sqrt{1-\delta^2}},  2\arctan\tfrac{\delta}{\sqrt{1-\delta^2}}\right]\,,
\end{align*}
which can be used to investigate the boundaries of $S_0^{\mathtt{inv}}(I)$ for any interval $I$ by the fact that
\begin{align*}
\bigcup_{\theta \in I} S_0^{\mathtt{inv}}(\{\theta\}) = S_0^{\mathtt{inv}}(I)\,.
\end{align*}
A direct computation of the derivative of the map in \eqref{eq-S0-1,delta-inverse-explicit} yields that for $\theta'$ defined in \eqref{eq-S0-1,delta-inverse-explicit} we have
\begin{align}
\mathsf{D}_\theta :=\frac{d\theta'}{d\theta} = \sqrt{1-\delta^2}\,\cdot \frac{|\tan \tfrac{\theta}{2}|}{\sqrt{\bigl(\tan \tfrac{\theta}{2}\bigr)^2+\delta^2}}\,.\label{eq-markov-transition-derivative}
\end{align}
Given the Poisson point process $\{(t_j, x_j) : t_1 < t_2 < \ldots\}$, denote by $\mathcal P_I$ the law of the embedded Markov chain of $\chl_t^{\mathtt{inv}}(I)$ and denote by $\widetilde{\mathcal P}_{|I|}$ the law of the embedded Markov chain of $\left|\chl_t^{\mathtt{inv}}(I)\right|$. Denote $I_k = \chl_{t_k}^{\mathtt{inv}}(I)$.  {We first state the behavior of the embedded interval chain. These properties identify its harmonic measure length and show that its limiting length is either \(0\) or \(2\pi\), with no further particle arrivals in the former case.}

\begin{proposition}\label{geometric properties of CHL}
    We have the following properties for the intervals $\{I_k\}_{k=0}^{\infty}$: 
    
    (1) $\{|I_k|\}_{k=0}^{\infty}$ is a positive martingale.

    (2) $|I_k|\in(0,2\pi)$.

    (3) $\lim_{k\to\infty}|I_k|\in\{0, 2\pi\}$ a.s.. In particular, if $\lim_{k\to\infty}|I_k|=0$, {there exists $k$ such that the interval receives no further particles after time $k$}.
\end{proposition}
\begin{proof}
    By our definition of $\{I_k\}_{k=1}^{\infty}$, it suffices to prove (1) and (3). First we prove (1). Define $$P(x) = \E\big[|I_{k+1}| \big| |I_k| = x \big].$$ Then $P(x)$ is nonnegative, $P(0) = 0$ and $P(2\pi) = 2\pi$. Moreover, by additivity to expectations, dividing the interval $I_k$ with length $x+y$ into two intervals with lengths $x,y$ yields $P(x+y)=P(x)+P(y)$ for $x,y \geq 0$ such that $x + y \leq 2\pi$. Therefore, $P(x)$ is linear and $P(x)=x$ by the boundary values at $x=0, 2\pi$, which completes our proof.

    For the first part of (3), we know that $\lim_{k\to\infty}|I_k|$ exists a.s. by (1). It suffices to prove that $\lim_{k\to\infty}|I_k|\in\{ 0,2\pi\}$ a.s.. Suppose on the contrary that there exists $0 < a < b < 2\pi$ such that $\P(a<\lim_{k\to\infty}|I_k|<b) > 0\,.$ Then $\bigcup_{\pi\delta^2/(10+0.5\delta^2)<h<2\pi}\bigl(h-\tfrac{1}{20}\delta^2(2\pi - h), h\bigr)$ is an {open cover} of $[a,b]$, so there exist $0<h_1<\ldots<h_s < 2\pi$ such that $(a,b) \subset \bigcup_{1 \leq k \leq s} \bigl(h_k-\tfrac{1}{20}\delta^2(2\pi - h_k), h_k\bigr)$. Therefore, we may assume $b-a \leq \tfrac{1}{10}\delta^2(2\pi - b)$ without loss of generality, and denote $r = \bigl(1-\sqrt{1-\delta^2}\bigr)(2\pi-b) > \tfrac{1}{10}\delta^2(2\pi- b)\,.$ Given any $I$ such that $|I| < b$, we have $\min_{x \in I}(|S_x^{\mathtt{inv}}(I)| - |I|) \geq r > b-a$ since 
    $$|S_x^{\mathtt{inv}}(I)| - |I| = \int_{I^c} \left(1-\mathsf{D}_\theta\right)d\theta \geq \bigl(1-\sqrt{1-\delta^2}\bigr)|I^c|$$
    by \eqref{eq-markov-transition-derivative}. Therefore, for $x \sim \mathsf{Unif}(\mathbb T_1)$ we have
    $$\P\big(|S_x^{\mathtt{inv}}(I)|\in(a,b)\big| |I|\in(a,b)\big) \leq \P\big(x \not\in I\big| |I|\in (a,b)\big) \leq 1-\tfrac{|I|}{2\pi} \leq 1 - \tfrac{a}{2\pi}\,,$$
    which yields for all $n \geq 1$ we have
    $$\lim_{m\rightarrow\infty} \P(|I_k|\in (a,b) \mbox{ for }n\leq k \leq n+m) \leq \lim_{m \rightarrow \infty} \bigl(1 - \tfrac{a}{2\pi}\bigr)^m = 0\,.$$
    Therefore $\P(a<\lim_{k\to\infty}|I_k|<b) = 0$, which is a contradiction.
    
    For the second part of (3), since $\lim_{k\to\infty}|I_k|=0$, there exists $k_0\in \mathbb{N}^*$ such that for any $k\ge k_0, |I_k|<\arctan\frac{\delta}{\sqrt{1-\delta^2}}.$ With (2) of Lemma \ref{lem-properties-mathcal-P}, we note that if there exists $k'>k$ such that the interval $I_{k'}$ will receive particles, then $|I_{k'}|\ge 4\arctan\frac{\delta}{\sqrt{1-\delta^2}}$, which leads to contradiction.
\end{proof}
Now we state the main quantitative properties of our Markov chain $\mathcal P_\cdot$, used in all our proofs.
\begin{lemma}\label{lem-properties-mathcal-P}
Suppose $\{ I_k\}_{k \geq 0} \sim \mathcal P_I$ for some closed interval $I \subset \mathbb T_1$. Then we have
\begin{enumerate}

\item \label{lem 3.2 1}For all $x\not\in I$,
\begin{align}\label{eq-estimation-x-notin-I}
\left|S_x^{\mathtt{inv}}(I)\right| \leq \sqrt{1-\delta^2} |I| \leq \bigl(1-\tfrac{1}{2}\delta^2\bigr)|I|.
\end{align}
As a corollary, with probability at least $1-\tfrac{|I|}{2\pi}$ we have $$|S_x^{\mathtt{inv}}(I)| \leq \sqrt{1-\delta^2} |I| \leq \bigl(1-\tfrac{1}{2}\delta^2\bigr)|I|.$$

\item \label{lem 3.2 -1}For all $x \in I$ we have $|S_x^{\mathtt{inv}}(I)| \geq |I|$.

\item \label{lem 3.2 0}For all $x\in \mathbb T_1$ we have $\big||S_x^{\mathtt{inv}}(I)| - |I|\big| \leq 4\arctan \frac{\delta}{\sqrt{1-\delta^2}}$.

\item \label{lem 3.2 2} If $|I_{k+1}| \geq |I_k|$, we have $$4\arctan\frac{\delta}{\sqrt{1-\delta^2}} \leq |I_{k+1}| \leq 4\arctan\frac{\delta}{\sqrt{1-\delta^2}} + \left(1 - \frac{1}{2}\delta^2\right)|I_k|.$$

\item \label{lem 3.2 3} If $|I_k| \leq \frac{\delta}{10}$, then there is probability at least $\frac{\delta-|I_k|}{2\pi}$ such that $|I_{k+1}| < \frac{1}{2}|I_k|.$

\item \label{lem 3.2 4} If $|I_k| \leq a \leq \frac{\delta}{10}$ for some positive constant $a$, then there is probability at least {$1-4\delta^{-1}a \geq \frac{3}{5}$} such that $\{|I_{k+j}|\}_{j \geq 0}$ is decreasing.

\item \label{lem 3.2 5} For $|I_0| \leq \frac{\delta}{10}$, denote 
\begin{align}\label{eq-def-sigma-star}
\sigma_{*}(I_0) = \inf\{ t \geq 1: |I_t| \leq \delta^3\mbox{ or } |I_t| > |I_{t-1}|\}\,,
\end{align}
then $\E_{\mathcal P_{I_0}}[\sigma_* (I_0)] \leq \delta^{-2}$.
\end{enumerate}
\end{lemma}
\begin{proof}
For (\ref{lem 3.2 1}), it follows directly from \eqref{eq-markov-transition-derivative}. 

For (\ref{lem 3.2 -1}), since $0 \not\in I^c$, we have $|S_x^{\mathtt{inv}}(I)| = 2\pi - \int_{I^c}\mathsf D_\theta d\theta \geq 2\pi - |\overline{I^c}| = |I|$.

For (\ref{lem 3.2 0}), note that when $x\not\in I$, we have $\big||S_x^{\mathtt{inv}}(I)| - |I|\big| =\big||S_x^{\mathtt{inv}}(\overline{I^c})| - |\overline{I^c}|\big|$ and $x \in \overline{I^c}$. Therefore, it suffices to deal with the $x \in I$ case, where we may denote the interval by $[a,b]$ for some $-\pi \leq a \leq 0 \leq b \leq \pi$. In this case, by (\ref{lem 3.2 -1}) we have
\begin{align*}
0 \leq |S_x^{\mathtt{inv}}(I)| - |I|&= \int_{[a,0)\cup(0,b]}\mathsf D_\theta d\theta + 4\arctan\frac{\delta}{\sqrt{1-\delta^2}} - |I|\\
&\leq |I| + 4\arctan\frac{\delta}{\sqrt{1-\delta^2}} - |I| =  4\arctan\frac{\delta}{\sqrt{1-\delta^2}}\,,
\end{align*}
which proves (\ref{lem 3.2 0}).

For (\ref{lem 3.2 2}), suppose $I_k = [a,b]$ and then we have $x_{k+1} \in [a,b]$, which yields $$\left[x_{k+1} -2\arctan\frac{\delta}{\sqrt{1-\delta^2}} ,  x_{k+1} + 2\arctan\frac{\delta}{\sqrt{1-\delta^2}}\right] = S_{x_{k+1}}^{\mathtt{inv}}(\{ x_{k+1}\}) \subset I_{k+1},$$ and therefore yields the lower bound. The upper bound comes from an application of (\ref{lem 3.2 1}) on $[a,x_{k+1})$ and $(x_{k+1},b]$.

For (\ref{lem 3.2 3}), denote {by} $\operatorname{Ex}(I)$ the interval of length $\delta$ with the same center as $I$ for any interval $I \subset \mathbb T_1$. Then by \eqref{eq-markov-transition-derivative} we have that whenever $|I_k| \leq \frac{\delta}{10}$ and $x_{k+1} \in \operatorname{Ex}(I_k) \setminus I_k$ we have
\begin{align}
 \frac{|I_{k+1}|}{|I_k|} \leq \max_{|\theta| \leq \delta }\left\{\sqrt{1-\delta^2}\,\cdot \frac{|\tan \tfrac{\theta}{2}|}{\sqrt{\bigl(\tan \tfrac{\theta}{2}\bigr)^2+\delta^2}} \right\} < \frac{1}{2}\,,
\end{align}
which yields (\ref{lem 3.2 3}) by the fact that $$\P_{\mathcal P} (x_{k+1} \in \operatorname{Ex}(I_k) \setminus I_k) \geq \frac{\delta - |I_k|}{2\pi}.$$

For (\ref{lem 3.2 4}), note that $$\P_{\mathcal P} ( |I_{k+1}| > |I_k|) = \P_{\mathcal P} (x_{k+1} \in  I_k) = \frac{|I_k|}{2\pi}\,.$$Therefore, combining the result in (3) we have
\begin{align}\label{lb-<1/2-or->1}
\P_{\mathcal P_I} \left( |I_{k+1}| > |I_k| \mbox{ or } |I_{k+1}| < \frac{1}{2}|I_k|\right) \geq \P_{\mathcal P} \left( |I_{k+1}| < \frac{1}{2}|I_k|\right) +\frac{|I_k|}{2\pi} \geq \frac{\delta}{2\pi}\,.
\end{align}
Next, define
\begin{align}
K &= \left\{ n : |I_{k+n}| > |I_{k+n-1}|\mbox{ or }|I_{k+n}| < \frac{1}{2}|I_{k+n-1}|\right\} \cup \{0\}= \{ 0 = n_0 < n_1  < \ldots\}\,,
\end{align}
then, by \eqref{lb-<1/2-or->1} we have
\begin{align*}
&\P_{\mathcal P_I}\left(|I_{k+n_l}| < \frac{1}{2}|I_{k+n_l-1}|\Big| |I_{k+n_l-1}| = \iota\right) \\
=\ & \sum_{v \geq 1}\P_{\mathcal P_I}(n_l = v)\cdot\P_{\mathcal P_I}\Big(|I_{k+v}| <\frac{1}{2}|I_{k+v-1}| \Big| |I_{k+v-1}|=\iota\,, n_l = v\Big)\\
=\ & \sum_{v \geq 1}\P_{\mathcal P_I}(n_l = v)\cdot\P_{\mathcal P_I}\Big(|I_{k+v}| <\frac{1}{2}|I_{k+v-1}| \Big| |I_{k+v-1}|=\iota\,, |I_{k+v}| <\frac{1}{2}|I_{k+v-1}| \mbox{ or }|I_{k+v}| >|I_{k+v-1}|\,;\\
\ & |K \cap \{x \in \mathbb Z : 0 \leq x \leq v - 1\} |= l\Big)\\
=\ & \sum_{v\geq 1}\P_{\mathcal P_I}(n_l = v)\cdot\tfrac{\P_{\mathcal P_I}\Big(|I_{k+v}| <\frac{1}{2}|I_{k+v-1}| \Big| |I_{k+v-1}|=\iota\Bigg)}{\P_{\mathcal P_I}\Big(|I_{k+v}| <\frac{1}{2}|I_{k+v-1}| \Big| |I_{k+v-1}|=\iota\Big ) + \P_{\mathcal P_I}\Big(|I_{k+v}|>|I_{k+v-1}| \Big| |I_{k+v-1}|=\iota\Big)} \\
\geq\ & \sum_{v\geq 1}\P_{\mathcal P_I}(n_l = v)\cdot \tfrac{\tfrac{\delta - \iota}{2\pi}}{\tfrac{\delta - \iota}{2\pi}+\tfrac{ \iota}{2\pi}} = \frac{\delta - \iota}{\delta}\,,
\end{align*}
since $\frac{\delta - \iota}{\delta}$ is decreasing with $\iota$, we have
\begin{align*}
&\P_{\mathcal P_I}\left(|I_{k+n_l}| < \frac{1}{2}|I_{k+n_l-1}|\Big| |I_{k+n_j}| < \frac{1}{2}|I_{k+n_j-1}|\mbox{ for }1 \leq j \leq l-1\right) \\
= \ & \sum_{v=1}\P_{\mathcal P_I}(n_l = v)\cdot\P_{\mathcal P_I}\left(|I_{k+n_l}| < \frac{1}{2}|I_{k+n_l-1}|\Big| |I_{k+n_j}| < \frac{1}{2}|I_{k+n_j-1}|\mbox{ for }1 \leq j \leq l-1\,, n_l = v\right) \\ 
\geq \ & \sum_{v=1}\P_{\mathcal P_I}(n_l = v)\cdot\min_{h \leq a/2^{l-1}}\P_{\mathcal P_I}\left(|I_{k+v}| < \frac{1}{2}|I_{k+v-1}|\Big| |I_{k+v-1}| = h\right) \geq \frac{\delta - \tfrac{a}{2^{l-1}}}{\delta}\,,
\end{align*}
therefore, we have
\begin{align}\label{eq-Markov-chain-estimate-(4)-1}
&\P_{\mathcal P_I}(|I_{k+j}|<|I_{k+j-1}| \mbox{ for all }j \geq 1)\nonumber \\
\geq\ & \prod_{l=1}^{+ \infty} \P_{\mathcal P_I}\left(|I_{k+n_l}| < \frac{1}{2}|I_{k+n_l-1}|\Big| |I_{k+n_j}| < \frac{1}{2}|I_{k+n_j-1}|\mbox{ for }1 \leq j \leq l-1\right) \nonumber \\
\geq\ & \prod_{l=1}^{\infty}\left(1 - \frac{a}{\delta\cdot 2^{l-1}}\right) > 1 - 4\delta^{-1}a \geq \frac{3}{5}\,,
\end{align}

{which yields (\ref{lem 3.2 4}).} 

For (\ref{lem 3.2 5}), we use the notation in the proof of (\ref{lem 3.2 4}) with $k = 0$. {By \eqref{lb-<1/2-or->1}, the increments $\{n_l - n_{l-1}\}_{l \geq 1}$ are stochastically dominated by independent geometric random variables with parameter $\frac{\delta}{2\pi}$. Therefore, by the stochastic domination of geometric variables we have}
\begin{align}
\,\quad\,\E[\sigma_*(I_0)] &\leq \delta^{-1}\log(\delta^{-1})\cdot\sum_{m \geq 1}\P_{\mathcal P_I}(\sigma_*(I_0) > m\delta^{-1}\log \delta^{-1})\nonumber\\
&\leq \delta^{-1}\log(\delta^{-1})\cdot\sum_{m \geq 1} \P_{\mathcal P_I}(n_{\lceil10\log \delta^{-1}\rceil} > m\delta^{-1}\log \delta^{-1})\nonumber \\
&= \delta^{-1}\log(\delta^{-1}) \cdot \sum_{m\geq1} \P\bigl(\mathsf{Bin}\bigl(m\delta^{-1}\log \delta^{-1},\tfrac{\delta}{2\pi}\bigr)<10\log \delta^{-1}\bigr) \nonumber\\
&{\leq \delta^{-1}\log(\delta^{-1})\left(200+\sum_{m\geq 200}2\exp\left(-\tfrac{m\log \delta^{-1}}{40 \pi}\right)\right) < \delta^{-2}\,,}\label{eq-Markov-chain-estimate-(4)-2}
\end{align}
{where in the last inequality we use a Bernstein estimate that gives
\begin{align*}
\P\Big(|\mathsf{Bin}(n,p)-np| \geq snp\Big) &\leq 2\exp\left(-\frac{nps^2}{4+s}\right)
\end{align*}
for $\bigl(n,p,s\bigr)=\bigl(m\delta^{-1}\log \delta^{-1},\tfrac{\delta}{2\pi},\tfrac{1}{2}\bigr)$. This proves (\ref{lem 3.2 5}).
}

\end{proof}

\subsection{Second moment estimate}
We next prove {several second moment bounds} of the interval variation.

\begin{lemma}\label{estimation of I(N-I)-1}
    Suppose $\{I_k\}_{k \geq 0} \sim \mathcal P_I$ for some closed interval $I$. Then there exist $\tilde{C_1}, \tilde{C_2}>0$ which are independent of $I_k$, such that
    \begin{equation}\label{eq:secondmoment}
    \begin{aligned}
        \tilde{C_1}\delta^{3}\le&\mathbb{E}_{\mathcal P_I}\left[(|I_{k+1}|-|I_{k}|)^2 \Big| |I_k|\right]\le \tilde{C_2}\delta^{3}, &  \delta\le |I_k|\le 2\pi-\delta,\\
        \tilde{C_1}\delta^{2}|I_k|\le&\mathbb{E}_{\mathcal P_I}\left[(|I_{k+1}|-|I_k|)^2 \Big| |I_k|\right]\le \tilde{C_2}\delta^{2}|I_k|,  & |I_k|< \delta ,\\
        \tilde{C_1}\delta^{2}(2\pi-|I_k|)\le&\mathbb{E}_{\mathcal P_I}\left[(|I_{k+1}|-|I_k|)^2 \Big| |I_k|\right]\le \tilde{C_2}\delta^{2}(2\pi-|I_k|), &\,  2\pi-\delta<|I_k|<2\pi,
    \end{aligned}
    \end{equation}
where $|I_{k+1}| \sim p(|I_{k}|,\cdot)$ is generated by the one-step Markov transition probability defined in \eqref{eq-markov-transition-inverse-slit-map}. {In particular, for $y \in (0, 2\pi)$ let $l(y):=S_0^{\mathtt{inv}}(y)-y$ and}
    {
    \begin{align*}
        m_\delta:=\frac1\pi\int_0^{2\pi}l(y)^2\,dy\,.
    \end{align*}}
    {Then we have $m_\delta =\frac{32}{3\pi}\delta^3-2\delta^4+O(\delta^5)$, and given $h \in \bigl[\delta^\alpha,2\pi-\delta^\alpha\bigr]$ for fixed $\alpha\in[0,1)$ we have}
    \begin{equation}\label{eq:secondmoment-sharp}
        {\mathbb{E}_{\mathcal P_I}\left[(|I_{k+1}|-|I_{k}|)^2 \Big| |I_k| = h\right]=m_\delta(1+o(1))}.
    \end{equation}
\end{lemma}
\begin{proof}
{We will first derive properties on $l(\cdot)$ and then we prove \eqref{eq:secondmoment} and \eqref{eq:secondmoment-sharp} separately. For $0<y<\pi$, we denote $r := r_y = \tan\bigl(\tfrac{y}{2}\bigr)$. Then, from \eqref{eq-S0-1,delta-inverse-explicit} we have}
\begin{align*}
{l(y)}&{=2\arctan A-2\arctan r,}\qquad
{l(2\pi-y)}{=-l(y),}
\end{align*}
{where $A=\sqrt{\frac{r^2+\delta^2}{1-\delta^2}} > r$. Since}
\begin{align*}
{A-r=\frac{\delta^2(1+r^2)}{(1-\delta^2)(A+r)},\qquad
\frac{A-r}{1+A^2} \leq \arctan A-\arctan r\le \frac{A-r}{1+r^2}.}
\end{align*}
{Therefore, for $0 < y < \pi$ we have 
\begin{align}\label{eq-l(y)-temp-ub}
0 <l(y) \leq \frac{2(A-r)}{1+r^2} \leq \frac{3\delta^2}{A+r} \leq \frac{3\delta^2}{r + \delta} < \frac{6\delta^2}{y+\delta}\,.
\end{align}
}
{Differentiating the same formula gives}
\begin{align*}
{l'(y)=\frac{\sqrt{1-\delta^2}\,r}{\sqrt{r^2+\delta^2}}-1 < 0,}
\end{align*}
{and hence, for $0 < y < \pi$ we have
\begin{align}
|l'(y)| = & \bigl(1-\frac{r}{\sqrt{r^2 + \delta^2}}\bigr) + \bigl(1-\sqrt{1-\delta^2}\bigr)\frac{r}{\sqrt{r^2 + \delta^2}}
\leq  \frac{\delta^2}{r^2 + \delta^2} + \delta^2\nonumber\\
\leq & \frac{\delta^2}{\bigl(\tfrac{y}{2}\bigr)^2 + \delta^2} + \frac{((2\pi)^2+1)\delta^2}{y^2+\delta^2}\leq  \frac{50\delta^2}{y^2+\delta^2}\,.\nonumber
\end{align}}
{By the fact that $l(2\pi - y) = -l(y)$, for $0 < y < 2\pi$ we get the symmetric bounds
\begin{align}\label{bound of l and derivative of l }
|l(y)|\le \frac{6\delta^2}{\delta+y\wedge(2\pi-y)}, \qquad 
|l'(y)|\leq \frac{50\delta^2}{(y\wedge(2\pi-y))^2+\delta^2}.
\end{align}}
{Moreover, for $0<y<\delta$, one has $r = \tan\bigl(\frac{y}{2}\bigr) \in \bigl(\frac{y}{2},y\bigr) \subset \bigl(0,\delta\bigr)$ and thus
\begin{align*}
\delta^2 \leq A^2 &= \frac{r^2+\delta^2}{1-\delta^2} \leq 2(y^2 + \delta^2) = 4\delta^2\,,\\
\arctan A-\arctan r &\geq \frac{A-r}{1+A^2}  \geq \frac{A-r}{2} \geq\frac{\delta^2(1+r^2)}{2(1-\delta^2)(A+r)} \geq\frac{\delta^2(1+r^2)}{2(2\delta+\delta)} > \frac{\delta}{6}\,,\\
\arctan A-\arctan r &\leq \frac{A-r}{1+r^2} \leq A-r\leq \frac{2\delta^2(1+r^2)}{(A+r)}\leq \frac{3\delta^2}{\delta + r} < 3\delta\,,
\end{align*}
therefore, we have
\begin{align}
\frac{\delta}{3} < l(y) < 6\delta\mbox{ for all } 0 < y < \delta\,.\label{eq:bound-l-when-a-<<-delta}
\end{align}
Finally, the upper bound (\ref{bound of l and derivative of l }) implies}
\begin{align}
\int_0^{2\pi}l(y)^2\,dy\le 72\int_0^\pi \frac{\delta^4}{(\delta+y)^2}\,dy\le 72\delta^3.\label{eq:delta-cube-ub-for-int-lsquare}
\end{align}

{Next, we prove \eqref{eq:secondmoment}. By rotation invariance, we take $I_k=[0,a]$ without loss of generality for some $a \in (0, 2\pi)$. Then, we have $|I_k| =a$ and the length increment caused by a slit at $x$ is}
{\begin{align*}
S_x^{\mathtt{inv}}([0,a])-a =l(x) - l\left(2\pi\left\{\frac{x-a}{2\pi}\right\}\right)
=\begin{cases}
        l(a-x)+l(x), &0\le x<a,\\
        l(x)-l(x-a), &a\le x<2\pi.
    \end{cases}
\end{align*}}
Therefore, we have
\begin{align}
\mathbb{E}_{\mathcal P_I}\left[(|I_{k+1}|-|I_{k}|)^2 \Big| |I_k| = a\right] &= \frac{1}{2\pi}\int_{0}^{2\pi}(S_x^{\mathtt{inv}}([0,a])-a)^2dx \nonumber\\
&= \frac{1}{2\pi}\int_{0}^{2\pi}\bigl(l(x) - l\left(2\pi\left\{\frac{x-a}{2\pi}\right\}\right)\bigr)^2 dx\,,\label{eq:second-moment-integral-rep}
\end{align}

{Based on \eqref{eq:second-moment-integral-rep}, we show the first line of \eqref{eq:secondmoment}. If $\delta\le a\le2\pi-\delta$, the upper bound follows immediately from 
\begin{align}
\eqref{eq:second-moment-integral-rep}
    \leq \frac{2}{2\pi}\int_{0}^{2\pi}\bigl(l(x)^2 + l\left(2\pi\left\{\frac{x-a}{2\pi}\right\}\right)^2\bigr)dx\overset{\eqref{eq:delta-cube-ub-for-int-lsquare}}{\leq }\frac{144}{\pi}\delta^3\,.\label{eq-ub-first-line-second-moment}
\end{align}

For the lower bound, by replacing $[0,a]$ with its complement if necessary, we may assume $\delta\le a\le\pi$ without loss of generality. Then, by \eqref{eq:bound-l-when-a-<<-delta} we have
\begin{align}
\eqref{eq:second-moment-integral-rep} \geq \frac{1}{2\pi}\int_{0}^{\delta }l(x)^2dx \overset{\eqref{eq:bound-l-when-a-<<-delta}}{\geq } \frac{1}{2\pi}\delta\cdot\bigl(\frac{\delta}{3}\bigr)^2 = \frac{1}{18\pi}\delta^3\,.\label{eq-lb-first-line-second-moment}
\end{align}
Combining \eqref{eq-ub-first-line-second-moment} and \eqref{eq-lb-first-line-second-moment}, we have shown the first line of \eqref{eq:secondmoment}.

Next we proceed to the second line of \eqref{eq:secondmoment}. If \(0<a<\delta\), then \eqref{eq:bound-l-when-a-<<-delta} gives \(l(y)\in (\frac{\delta}{3}, 6\delta)\) for \(0<y<\delta\). Since both \(x\) and \(a-x\) lie in \(\bigl(0,\delta\bigr)\) when $0 < x < a$, we have \(l(a-x)+l(x)\in(\frac{2\delta}{3}, 12\delta)\), which gives
\begin{align}
\int_{[0,a]}\bigl(|S_x^{\mathtt{inv}}([0,a])|-a\bigr)^2\,dx = \int_{[0,a]} (l(a-x)+l(x))^2 dx\in \left[\frac{4a\delta^2}{9}\,, 144a\delta^2\right]\,.\label{eq-[0,a]-int-bound}
\end{align}
On $[0,a]^c$, (\ref{bound of l and derivative of l }) gives
\begin{align}
0\leq\int_{[0,a]^c}\bigl(|S_x^{\mathtt{inv}}([0,a])|-a\bigr)^2\,dx &= \int_{[0,a]^c} (l(x) - l(x-a))^2 dx \nonumber\\
&\le a^2\int_0^{2\pi}\left(\frac{50\delta^2}{(y\wedge (2\pi - y))^2+\delta^2}\right)^2dy\nonumber \\
&= 5000a^2\biggl(\frac{\delta^2}{2}\cdot \frac{\pi}{\pi^2+\delta^2} + \frac{\delta}{2}\arctan\bigl(\frac{\pi}{\delta}\bigr)\biggr)\nonumber\\
&< 5000a^2\delta\pi < 20000a\delta^2\,.\label{eq-[0,a]-complement-int-bound}
\end{align}
Combining \eqref{eq-[0,a]-int-bound} and \eqref{eq-[0,a]-complement-int-bound} proves the second line of \eqref{eq:secondmoment}; the third follows by applying the same argument to the complementary interval of length $2\pi-a$.}

{We now prove \eqref{eq:secondmoment-sharp}. Note that $\int_0^{2\pi}l(y)^2\,dy$ can be represented in the following form by substitution of $y$ by $2\arctan r$:}
\begin{align}
\int_0^{2\pi}l(y)^2\,dy=16\int_0^\infty \frac{1}{1+r^2}\left(\arctan\sqrt{\frac{r^2+\delta^2}{1-\delta^2}}-\arctan r\right)^2dr.\label{eq-tan-rep-of-int-l^2}
\end{align}
{Splitting the integral into $r\le \delta$ and $r\ge\delta$, Taylor expansion of $\arctan$ gives}
\begin{align*}
&\arctan\sqrt{\frac{r^2+\delta^2}{1-\delta^2}}-\arctan r\in \left[\frac{\sqrt{\tfrac{r^2+\delta^2}{1-\delta^2}}-r}{\tfrac{r^2+1}{1-\delta^2}}\,, \frac{\sqrt{\tfrac{r^2+\delta^2}{1-\delta^2}}-r}{r^2+1}\right]\\
=\ & \left[\frac{\delta^2}{\sqrt{\tfrac{r^2+\delta^2}{1-\delta^2}}+r}\,, \frac{\tfrac{\delta^2}{1-\delta^2}}{\sqrt{\tfrac{r^2+\delta^2}{1-\delta^2}}+r}\right] \subset \left[\frac{\delta^2\sqrt{1-\delta^2}}{\sqrt{r^2+\delta^2}+r}\,, \frac{\delta^2(1-\delta^2)^{-1}}{\sqrt{r^2+\delta^2}+r}\right]
\end{align*}
{Therefore, we have
\begin{align*}
\arctan\sqrt{\frac{r^2+\delta^2}{1-\delta^2}}-\arctan r = \frac{\delta^2}{\sqrt{r^2+\delta^2}+r}\bigl(1+O(\delta^2)\bigr)= \bigl(\sqrt{r^2+\delta^2}-r\bigr)\bigl(1+O(\delta^2)\bigr)\,.
\end{align*}
}
{Therefore, by the substitution $r=\delta\sinh u$,}
\begin{align*}
{\int_0^{2\pi}l(y)^2\,dy}
&{=16\bigl(1+O(\delta^2)\bigr)^2\int_0^\infty \frac{\bigl(\sqrt{r^2+\delta^2}-r\bigr)^2}{1+r^2}\,dr}\\
&{=\bigl(1+O(\delta^2)\bigr)\bigl(32\delta-32\sqrt{1-\delta^2}\arcsin\delta-8\pi\bigl(1-\sqrt{1-\delta^2}\bigr)^2\bigr)}\\
&{=\frac{32}{3}\delta^3-2\pi\delta^4+O(\delta^5)\,,}
\end{align*}
{which proves the expansion of $m_\delta$. Next, suppose that $h \in \bigl[\delta^\alpha,2\pi-\delta^\alpha\bigr]$ for some $\alpha \in [0,1)$. Then, by (\ref{bound of l and derivative of l }) we have
\begin{align*}
\Big|\int_{[0,h]}l(h-x)l(x) dx\Big|\leq \int_{[0,h]}\frac{36\delta^4 dx}{(\delta+x)(\delta + h - x)}\leq \frac{72\delta^4}{2\delta + h} \int_{[\delta, h+\delta]}\frac{dx}{x} = O(\delta^{4-\alpha}\log(\delta^{-1}))\,,
\end{align*}
and similarly, for $h < x < \pi+\frac{h}{2}$ we have $|l(x)l(x-h)| \leq \frac{36\delta^4}{(\delta + x-h)(\delta + h)}$ and thus
\begin{align*}
\Big|\int_{[0,h]^c}l(x-h)l(x)dx\Big| \leq 72\delta^4\int_{\bigl[h,\pi + \tfrac{h}{2}\bigr]} \frac{1}{(\delta + x-h)(\delta + h)} = O(\delta^{4-\alpha}\log(\delta^{-1}))\,,
\end{align*}
 Thus, we have
\begin{align*}
&\E_{\mathcal P_I}\Big[(|I_{k+1}|-|I_k|)^2\Big| |I_k| = h\Big] =\frac{1}{2\pi}\int_{[0,2\pi]}\bigl(|S_x^{\mathtt{inv}}([0,h])|-h\bigr)^2\,dx \\
=\ & \frac{1}{2\pi}\int_{[0,h]}\bigl(l(h-x)+l(x)\bigr)^2\,dx + \frac{1}{2\pi}\int_{[0,h]^c}\bigl(l(x)-l(x-h)\bigr)^2\,dx \\
=\ & m_\delta + \frac{1}{\pi}\int_{[0,h]}l(h-x)l(x)dx - \frac{1}{\pi}\int_{[0,h]^c}l(x-h)l(x)dx \\
=\ & m_\delta + O(\delta^{4-\alpha}\log(\delta^{-1})) = m_\delta(1+o(1))\,,
\end{align*}
completing the proof of \eqref{eq:secondmoment-sharp}.}
\end{proof}

Next, given $\{I_k\}_{k \geq 0} \sim \mathcal P_{I_0}$ for a given closed interval $I_0 \subset \mathbb T_1$, we state the following lemma to estimate the average time needed for an interval to become very small or large under $\mathcal P$. Define the stopping time
\begin{align*}
    \sigma_{a,b}(I_0)&=\inf\left\{k\in \mathbb{N}^*: |I_k|<a\ \text{or}\ |I_k|>b\right\}
\end{align*}
with $0< a\le  b\le 2\pi$
and
$$
    g_{a,b}(x)=\mathbb{E}_{\mathcal P_{I_0}}\left[\sigma_{a,b}(I_0)\Big||I_0|=x\right]
$$
with $g_{a,b}(x)=0$ when $0<x\le a, b\le x<2\pi$. 

{Define
\[
    v_\delta(h):=\mathbb{E}_{\mathcal P_{I_0}}\left[\left(|I_{k+1}|-|I_k|\right)^2\,\middle|\, |I_k|=h\right].
\]}
{We use the following exit estimate for the interval-length martingale $\{|I_k|\}_{k\ge0}$.}
\begin{lemma}\label{lem:quadratic-exit-estimate}
{For all sufficiently small $\delta$, we have
\begin{equation}\label{eq:one-step-jump-bound}
    \big||I_{k+1}|-|I_k|\big|\le 5\delta \qquad \text{a.s. for all }k.
\end{equation}
Moreover, let $0<a<x<b<2\pi$. If there exists $0<v_-\le v_+<\infty$ such that $ v_-\le v_\delta(h)\le v_+$ for all $h\in(a,b)$,
then
\begin{equation}\label{eq:quadratic-exit-estimate}
    \frac{(x-a)(b-x)}{v_+}
    \le g_{a,b}(x)
    \le \frac{(x-a+5\delta)(b-x+5\delta)}{v_-}.
\end{equation}}
\end{lemma}
\begin{proof}
{\eqref{eq:one-step-jump-bound} follows directly from Lemma~\ref{lem-properties-mathcal-P}(3). To prove \eqref{eq:quadratic-exit-estimate}, let $Q(y):=(y-a)(b-y)$. Since \((|I_k|)_{k\ge0}\) is a martingale, on $\{a<|I_k|<b\}$ we have
\[
    \mathbb{E}_{\mathcal P_{I_0}}\left[Q(|I_{k+1}|)-Q(|I_k|)\,\middle|\, |I_k|\right]=-v_\delta(|I_k|).
\]
Applying optional stopping theorem to
\[
    Q(|I_{k\wedge\sigma_{a,b}(I_0)}|)
    +\sum_{j<k\wedge\sigma_{a,b}(I_0)}v_\delta(|I_j|)
\]
and then letting $k\to\infty$ gives
\[
    \mathbb{E}_{\mathcal P_{I_0}}\Big[\sum_{j<\sigma_{a,b}(I_0)}v_\delta(|I_j|)\Big]
    =Q(x)-\mathbb{E}_{\mathcal P_{I_0}}[Q(|I_{\sigma_{a,b}(I_0)}|)].
\]
Here $Q(|I_{\sigma_{a,b}(I_0)}|)\le0$. Hence, by $v_\delta\le v_+$ we have $g_{a,b}(x)\ge Q(x)/v_+$, which is the lower bound. For the upper bound, \eqref{eq:one-step-jump-bound} implies $|I_{\sigma_{a,b}(I_0)}|\in[a-5\delta,a)\cup(b,b+5\delta]$, and so
\[
    -Q(|I_{\sigma_{a,b}(I_0)}|)\le 5\delta(b-a+5\delta).
\]
Using $v_\delta\ge v_-$, we obtain
\[
    v_-g_{a,b}(x)
    \le Q(x)+5\delta(b-a+5\delta)
    =(x-a+5\delta)(b-x+5\delta),
\]
which proves the claim.}
\end{proof}

Based on Lemma \ref{lem:quadratic-exit-estimate}, we propose the following estimate of $g_{a,b}(x)$: 
\begin{lemma}\label{lem-harmonic-estimations-quadratic}
    {For $a<x<b$,}
\begin{align}\label{eq-g_a,b-lb}
        {g_{a,b}(x)\ge \frac{(x-a)(b-x)}{\tilde{C_2}\delta^3}.}
\end{align}
    {Moreover, let $\sigma'(I_0)=\sigma_{\delta^3,2\pi-\delta^3}(I_0)$. If $\frac{\delta}{10} < h \wedge (2\pi - h) \leq 5\delta$, then there exists a  universal constant $C_*$ such that for all sufficiently small $\delta$,}
 \begin{align}\label{eq:bound-sigma-prime-expectation}
      \mathbb{E}_{\mathcal P_{I_0}}\left[\sigma'(I_0)\Big||I_0|= h\right]\le C_*\delta^{-2}.
 \end{align}
\end{lemma}
\begin{proof}
By Lemma~\ref{estimation of I(N-I)-1}, for all $h \in (0,2\pi)$ we have 
\begin{align*}
\E_{\mathcal P_{I_0}}\Big[ (|I_{k+1}| - |I_k|)^2 \Big| |I_k| = h\Big] \leq \tilde{C_2}\delta^3
\end{align*}
for some constant $\tilde{C_2} > 0$. Therefore, \eqref{eq-g_a,b-lb} follows from Lemma~\ref{lem:quadratic-exit-estimate} with $v_+=\tilde{C_2}\delta^3$.

It remains to prove \eqref{eq:bound-sigma-prime-expectation}. Define $H = \{h: \frac{\delta}{10}< h \wedge (2\pi - h ) \le 5\delta \}$. Define 
\[
\eta_{I_0}=\inf\{k\ge0: |I_k| \wedge (2\pi - |I_k|)\le\frac{\delta}{10}\}\,.
\]
Also, since $\E_{\mathcal P_{I_0}}[\sigma'(I_0)]$ only depends on $|I_0|$ by symmetry, we define
\[
\Lambda(|I_0|) = \E_{\mathcal P_{I_0}}[\sigma'(I_0)]\,.
\]
Note that by Lemma~\ref{estimation of I(N-I)-1} we have
\begin{align*}
\E_{\mathcal P_{I_0}}\Big[ (|I_{k+1}| - |I_k|)^2 \Big| |I_k| = h\Big] \geq \frac{\tilde{C_1}\delta^3}{10}\mbox{ for all }h \in H.
\end{align*}
Therefore, applying Lemma~\ref{lem:quadratic-exit-estimate} with $a=\frac{\delta}{10}$ and $b=2\pi-\frac{\delta}{10}$ gives
\begin{align}\label{eq:boundary-layer-eta}
    \mathbb{E}_{\mathcal P_{I_0}}\bigl[\eta_{I_0}\bigr]\le \frac{(5\delta + 5\delta)(2\pi + 5\delta)}{\tfrac{\tilde{C_1}\delta^3}{10}} \leq C\delta^{-2}.
\end{align}
for some constant $C$. Next, define
\begin{align*}
\theta_{I_0} = \inf\{k \geq 1: &|I_{\eta_{I_0}+k}| \wedge (2\pi - |I_{\eta_{I_0}+k}|) \leq \delta^3 \mbox{ or } \\
&|I_{\eta_{I_0}+k}| \wedge (2\pi - |I_{\eta_{I_0}+k}|) > |I_{\eta_{I_0}+k-1}| \wedge (2\pi - |I_{\eta_{I_0}+k-1}|)\}\,,
\end{align*}
Then, by applying Item (4) of Lemma~\ref{lem-properties-mathcal-P} to the shorter interval between $I_k$ and $\overline{I_k^c}$, we have 
\begin{align*}
|I_{\eta_{I_0} + \theta_{I_0}}| \wedge (2\pi - |I_{\eta_{I_0} + \theta_{I_0}}|) &\in [0,\delta^3] \cup \left[4\arctan\frac{\delta}{\sqrt{1-\delta^2}} \,, 4\arctan\frac{\delta}{\sqrt{1-\delta^2}}  + \frac{\delta}{10}\right]\\
&\subset[0,\delta^3] \cup [3\delta, 5\delta]
\end{align*}
when $\delta$ is small enough. Also, by Items (6) and (7) of Lemma~\ref{lem-properties-mathcal-P} to the shorter interval between $I_k$ and $\overline{I_k^c}$, we have
\begin{align}
\P_{\mathcal P_{I_0}}\big[ |I_{\eta_{I_0} + \theta_{I_0}}| \wedge (2\pi - |I_{\eta_{I_0} + \theta_{I_0}}|) \in [3\delta,5\delta] \big] \leq \frac{2}{5}\,,\label{bound-theta-I_0-prob}\\ 
\E_{\mathcal P_{I_0}}\bigl[\theta_{I_0}\bigr] < \delta^{-2}\,,\label{bound-theta-I_0-mean}
\end{align}
\begin{align*}
\Lambda(|I_0|) &\overset{\eqref{bound-theta-I_0-prob}}{\leq} \mathbb{E}_{\mathcal P_{I_0}}\bigl[\eta_{I_0}\bigr] + \E_{\mathcal P_{I_0}}\bigl[\theta_{I_0}\bigr] + \frac{2}{5}\sup_{h \wedge (2\pi - h) \in [3\delta , 5\delta]} \Lambda(h)\\
&\overset{\eqref{eq:boundary-layer-eta},\eqref{bound-theta-I_0-mean}}{\leq} (C+1)\delta^{-2} + \frac{2}{5}\sup_{h \in H} \Lambda(h)\,.
\end{align*}
Taking supremum over $|I_0| \in H$, we have proved \eqref{eq:bound-sigma-prime-expectation} as desired.
\end{proof}

\begin{corollary}\label{cor:lem-harmonic-estimations-quadratic}
    {Fix $\alpha\in[0,1)$ and let $m_\delta$ be as in Lemma \ref{estimation of I(N-I)-1}. There exists $\varepsilon_\delta=\varepsilon_\delta(\alpha)\downarrow0$ as $\delta\downarrow 0$ such that, for $\delta^\alpha\le a<x<b\le2\pi-\delta^\alpha$ and $(x-a)\wedge(b-x)\ge\delta^\alpha$,}
    $$
        {\frac{(x-a)(b-x)}{m_\delta(1+\varepsilon_\delta)}\le g_{a,b}(x)\le \frac{(x-a)(b-x)}{m_\delta(1-\varepsilon_\delta)}.}
    $$
\end{corollary}
\begin{proof}
{Let $\eta_\delta=\eta_\delta(\alpha)\downarrow0$ be such that, uniformly for $h\in[\delta^\alpha,2\pi-\delta^\alpha]$,}
\[
    {m_\delta(1-\eta_\delta)\le v_\delta(h)\le m_\delta(1+\eta_\delta).}
\]
{This is exactly the uniform form of \eqref{eq:secondmoment-sharp}. Since $(a,b)\subset[\delta^\alpha,2\pi-\delta^\alpha]$, \eqref{eq:quadratic-exit-estimate} applies with $v_-=m_\delta(1-\eta_\delta)$ and $v_+=m_\delta(1+\eta_\delta)$. The lower bound follows immediately. For the upper bound, \eqref{eq:one-step-jump-bound} and $(x-a)\wedge(b-x)\ge\delta^\alpha$ give}
\[
    {(x-a+5\delta)(b-x+5\delta)
    \le (x-a)(b-x)\left(1+C\delta^{1-\alpha}\right).}
\]
{Therefore, $\varepsilon_\delta = \eta_\delta + C\delta^{1-\alpha}(1+\eta_\delta) \downarrow 0$ satisfies our claim.}
\end{proof}


\section{Expectation of one-arm domination time}\label{sec:Expectation of one-arm domination time}
{This section proves Theorem \ref{thm: one-arm domination time}. To give the lower bound, we compare one-arm domination with the time at which one of two semicircular inverse intervals becomes macroscopic. To give the upper bound, we introduce ``marked configurations" to track the inverse harmonic measure intervals of all trees, reducing one-arm domination to one-interval domination, where we can apply the estimates developed in Section \ref{sec:Estimate of Interval variation}.}

\subsection{Lower bound}\label{sec: Lower bound}
We first estimate the lower bound.
Consider $I:=[0,\pi)$ and $I':=\mathbb T_1\backslash I$. Given the Poisson point process $\{(t_i,x_i) : t_1 < t_2 < \ldots\}$. Denote $I_{k} := \chl_{t_k}^{\mathtt{inv}}(I), I'_{k} := \chl_{t_k}^{\mathtt{inv}}(I')$.
{Define the events $A^{(k)}_1$ and $A^{(k)}_2$ by}
\begin{align*}
     A^{(k)}_1:=\{  x_{k'+1}\notin I_{k'}, \forall  k'\ge k\},\qquad 
     A^{(k)}_2:=\{  x_{k'+1}\notin I_{k'}', \forall k'\ge k\}.
\end{align*}
Denote $T$ as
\[
    T=\inf\{k: A^{(k)}_1 \cup A^{(k)}_2 \ \text{occurs}  \}.
\]

When there is a one-arm domination, the dominating tree is either on $I$ or on $I'$. Therefore, $\upsilon_{N, \lambda}\ge t_T$ always holds. Note that 
\begin{align}\label{eq:t_T and T}
    \mathbb{E}[t_T]&=\mathbb{E}[\mathbb{E}[t_T|T]]=\mathbb{E}\left[\sum_{k=0}^{\infty}t_k\mathbb{P}(T=k)\right]=\sum_{k=0}^{\infty}\mathbb{E}[t_k]\mathbb{P}(T=k)=\sum_{k=0}^{\infty}\frac{k}{2\pi N}\mathbb{P}(T=k)\nonumber \\
    &=\frac{\mathbb{E}[T]}{2\pi N}.
\end{align}
Thus, it suffices to lower bound $\E[T]$. We shall lower bound $\E[T]$ by Lemma~\ref{lem-lower-bound-ET}:
\begin{lemma}\label{lem-lower-bound-ET}
    There exists $c'>0$ such that
    \[
        \E[T]>c' \delta^{-3}.
    \]
\end{lemma} 
\begin{proof}
    Define the stopping time
\begin{align}\label{stopping time-sigma'}
    \sigma'&=\inf\left\{k\in \mathbb{N}^*: |I_k|<\frac{\pi}{3}\ \text{or}\ |I_k|>\frac{5\pi}{3}\right\}.
\end{align}
By Lemma \ref{lem-harmonic-estimations-quadratic} we obtain that
\begin{align*}
    \mathbb{E}[\sigma']\ge \frac{4\pi^2}{9\tilde{c}\delta^3}.
\end{align*}
Denote the event 
\begin{align}\label{event A^{(sigma')}}
     A^{(\sigma')}:=\{x_{\sigma'+1}\in I_{\sigma'}, x_{\sigma'+2}\in I'_{\sigma'+1}\}.
\end{align}
For large $N$, by strong Markov property and Item (3) of Lemma \ref{lem-properties-mathcal-P} we have 
\begin{align*}
\frac{|I_{\sigma'}|(2\pi - |I_{\sigma'+1}|)}{4\pi^2} \geq \frac{\bigl(\tfrac{\pi}{3}-O(\delta)\bigr)\cdot \bigl(\tfrac{5\pi}{3} - O(\delta)\bigr)}{4\pi^2} \geq \frac{1}{8}
\end{align*}
holds deterministically, which yields that there exists a constant $\tilde{c}>0$ such that 
\begin{equation}
\begin{aligned}
    {\mathbb{E}[\sigma' \mathbf{1}_{A^{(\sigma')}}]=\mathbb{E}\left[\sigma'\,\mathbb{P}\left(A^{(\sigma')}\mid \mathcal F_{\sigma'}\right)\right]\ge \frac{1}{8}\mathbb{E}[\sigma']\ge \frac{\pi}{18\tilde{c}\delta^{3}}.}
\end{aligned}
\end{equation}
This proves the lemma since $T\ge \sigma' \mathbf{1}_{A^{(\sigma')}}$.
\end{proof}
Therefore, by \eqref{eq:t_T and T} we obtain the lower bound of $\mathbb{E}[\upsilon_{N, \lambda}]$ in Theorem \ref{thm: one-arm domination time}, i.e. there exists $c>0$ such that
\begin{align*}
    \mathbb{E}[\upsilon_{N, \lambda}]>c\lambda^{-3}N^2.
\end{align*}

\subsection{Upper bound}\label{upper bound}
{We now prove the upper bound} based on our analysis of $\mathcal P_I$ for some $I \subset \mathbb T_1$. The key is to consider a series of marked intervals on $\mathbb T_1$ and introduce a Markov chain on a set of marked intervals on a unit circle. Then, we couple the set of these marked intervals with the hitting probabilities of all trees in the actual $\text{CHL}_N$ process, {so that} the marginal law of each marked interval is $\mathcal P_I$ for some $I \subset \mathbb T_1$. We now introduce the definition and generation of our marked intervals:
\begin{definition}[Marked configuration]\label{defn:Mark configuration}
    A marked configuration $\mathtt M$ on $\mathbb T_1$, denoted by \\$[(I_1 , I_2 , \ldots , I_k), (a_1 , \ldots , a_k)]$, consists of an ordered tuple of closed intervals and another ordered tuple indicating the corresponding colors of each interval. The colors $a_1 , \ldots , a_k$ are assigned to be non-zero integers. For $1 \leq j \leq k$, we say $I_j$ is marked by color $a_j$ in the marked configuration $\mathtt M =[(I_1 , I_2 , \ldots , I_k), (a_1 , \ldots , a_k)]$. Also, we denote the set $\overline{\mathbb T_1 \setminus (\bigcup_{j=1}^{k} I_j)}$ by $I_0 = I_0(\mathtt M)$. Then $I_0(\mathtt M)$ is a union of disjoint closed intervals $J_1 , \ldots, J_l$ and we say $J_j$ is marked by color $0$ for $ 1\leq j \leq l$. 
    
    For such $\mathtt M$, we define the number of colors of $\mathtt M$ by $\mathtt n(\mathtt M) = k$. Also, we define $\mathtt M(a_j) = I_j$ for $1 \leq j \leq k$, $\mathtt M(0) = I_0$, and $\mathtt M(a) = \emptyset$ for any unused color $a$. When $\mathtt n(\mathtt M)=0$, $\mathtt M$ consists of a unique interval with color $0$ and we say $\mathtt M = \emptyset$. 
    
    For $1 \leq j \leq k$ we say that $x \in \mathbb T_1$ is an endpoint of color $j$ if $x \in \partial I_j$.
\end{definition}
\begin{definition}[Marked configuration sequence coupled to a $\text{CHL}_1$]\label{defn:Mark configuration sequence coupled}
    Given a $\text{CHL}_1$ on $\mathbb T_1 \times \mathbb R_+$ with the Poisson point process $\mathcal A = \{ (t_j, x_j): t_j\mbox{ is increasing}\}$, define the marked configuration sequence $\{ \mathtt M_j \}_{j \geq 0}$ coupled to the $\text{CHL}_1$ recursively as follows. 
\begin{enumerate}
   \item $\mathtt M_0 = \emptyset$;
    \item $\mathtt M_{j+1} = S_{x_{j+1}}^{\mathtt{inv}} (\mathtt M_{j})$, where for $\mathtt M_{j} = [(I_1 ,\ldots , I_k),(1,\ldots , k)]$
     \begin{equation}\label{formula of Mark configuration sequence} S_{x_{j+1}}^{\mathtt{inv}} (\mathtt M_{j}) =
    \begin{cases}
    [(S_{x_{j+1}}^{\mathtt{inv}}(I_1) ,  \ldots , S_{x_{j+1}}^{\mathtt{inv}}(I_k)), (1,\cdots,k)],\  \exists 1\le l\le k, x_{j+1} \in I_l \,,  \\
    {[(S_{x_{j+1}}^{\mathtt{inv}}(I_1) , \ldots , S_{x_{j+1}}^{\mathtt{inv}}(I_k), S_{x_{j+1}}^{\mathtt{inv}}(\{x_{j+1}\})), (1,\cdots,k+1)],}\  \mbox{else}\,.
    \end{cases}
    \end{equation}
    \end{enumerate}
     For each color $a$ we define $\mathsf{Gen}_a = \inf\{ j : \mathtt M_j(a) \neq \emptyset\}$ the generating time of color $a$. 
    \end{definition}
     By construction, $\{\mathtt M_j\}_{j \geq 0}$ is a Markov process, and for {every} color $a$, $\{ \mathtt M_{j}(a)\}_{j \geq \mathsf{Gen}_a}$ {has} law of $\mathcal P$. Also, since the process is only determined by $\{ x_j\}_{j \geq 1}$, we also say $\{\mathtt M_j \}_{j \geq 0}$ is \textit{the marked configuration sequence generated by $\{ x_j \}_{j \geq 1}$}.
    For a marked configuration sequence $\{ \mathtt M_j \}_{j \geq 0}$, define the $l$-th growing color $\mathtt C_l(\{ \mathtt M_j \}_{j \geq 0})$ by the unique color $k$ such that $0 \neq |\mathtt M_l(k)| \geq |\mathtt M_{l-1} (k)|$.  {Figure~\ref{fig:marked-configuration-sequence} illustrates the two possible updates of $\mathtt M_j$: growth of an existing non-zero color and creation of a new color from the zero-colored set. In the first case we have $\mathtt C_{j+1}(\{\mathtt M_j\}_{j \geq 0}) = 2$, and in the second case we have $\mathtt C_{j+1}(\{\mathtt M_j\}_{j \geq 0}) = 4$.}

    \begin{figure}[htbp]
    \centering
    \begin{tikzpicture}[
        x=1cm,y=1cm,
        >=Stealth,
        font=\small,
        zeroarc/.style={line width=4.8pt, draw=gray!28, line cap=round},
        colone/.style={line width=4.8pt, draw=blue!80!black, line cap=round},
        coltwo/.style={line width=4.8pt, draw=orange!90!black, line cap=round},
        colthree/.style={line width=4.8pt, draw=green!50!black, line cap=round},
        colnew/.style={line width=4.8pt, draw=red!75!black, line cap=round},
        endpoint/.style={circle, draw=black, fill=white, inner sep=0.9pt},
        hitpoint/.style={circle, draw=black, fill=black, inner sep=1.3pt}
    ]
    \def\R{1.15}
    
    \node at (3.4,2.4) {$x_{j+1}\in \mathtt M_j(l)$: an existing color is updated};
    
    \begin{scope}[shift={(0,0)}]
        \draw[zeroarc] (0,0) circle[radius=\R];
        \draw[colone]   (25:\R)  arc[start angle=25,end angle=95,radius=\R];
        \draw[coltwo]   (145:\R) arc[start angle=145,end angle=215,radius=\R];
        \draw[colthree] (260:\R) arc[start angle=260,end angle=315,radius=\R];
    
        \foreach \a in {25,95,145,215,260,315}
            \node[endpoint] at (\a:\R) {};
    
        \node[hitpoint] at (180:\R) {};
    
        \node[text=blue!80!black]   at (70:1.55) {$\mathtt M_j(1)=I_1$};
        \node[text=orange!90!black] at (180:2.35) {$x_{j+1}\in \mathtt M_j(2)$};
        \node[text=green!50!black]  at (292:1.55) {$\mathtt M_j(3)$};
        \node[text=gray!70!black]   at (345:1.68) {$\mathtt M_j(0)$};
    
        \node at (0,-1.95) {$\mathtt M_j=[(I_1,I_2,I_3),(1,2,3)]$};
    \end{scope}
    
    \draw[->,thick] (2.6,0) -- (4,0)
        node[midway,above] {$S_{x_{j+1}}^{\mathtt{inv}}$};
    
    \begin{scope}[shift={(6.7,0)}]
        \draw[zeroarc] (0,0) circle[radius=\R];
        \draw[colone]   (18:\R)  arc[start angle=18,end angle=88,radius=\R];
        \draw[coltwo]   (128:\R) arc[start angle=128,end angle=232,radius=\R];
        \draw[colthree] (265:\R) arc[start angle=265,end angle=318,radius=\R];
    
        \foreach \a in {18,88,128,232,265,318}
            \node[endpoint] at (\a:\R) {};
    
        \node at (0,1.72) {$\mathtt n(\mathtt M_{j+1})=3$};
        \node[text=orange!90!black] at (165:1.95) {$\mathtt C_{j+1}=2$};
        \node at (0,-1.95)
        {$\mathtt M_{j+1}(a)=S_{x_{j+1}}^{\mathtt{inv}}(\mathtt M_j(a)),\ a=1,2,3$};
    \end{scope}
    
    \node at (3.4,-3.0) {$x_{j+1}\in \mathtt M_j(0)$: new color is generated};
    
    \begin{scope}[shift={(0,-5.4)}]
        \draw[zeroarc] (0,0) circle[radius=\R];
        \draw[colone]   (25:\R)  arc[start angle=25,end angle=95,radius=\R];
        \draw[coltwo]   (145:\R) arc[start angle=145,end angle=215,radius=\R];
        \draw[colthree] (260:\R) arc[start angle=260,end angle=315,radius=\R];
    
        \foreach \a in {25,95,145,215,260,315}
            \node[endpoint] at (\a:\R) {};
    
        \node[hitpoint] at (340:\R) {};
    
        \node[text=gray!70!black] at (335:2.1) {$x_{j+1}\in \mathtt M_j(0)$};
        \node at (0,-1.95) {$\mathtt M_j=[(I_1,I_2,I_3),(1,2,3)]$};
    \end{scope}
    
    \draw[->,thick] (2.6,-5.4) -- (4,-5.4)
        node[midway,above] {$S_{x_{j+1}}^{\mathtt{inv}}$};
    
    \begin{scope}[shift={(6.7,-5.4)}]
        \draw[zeroarc] (0,0) circle[radius=\R];
        \draw[colone]   (18:\R)  arc[start angle=18,end angle=88,radius=\R];
        \draw[coltwo]   (145:\R) arc[start angle=145,end angle=215,radius=\R];
        \draw[colthree] (255:\R) arc[start angle=255,end angle=310,radius=\R];
        \draw[colnew]   (325:\R) arc[start angle=325,end angle=352,radius=\R];
    
        \foreach \a in {18,88,145,215,255,310,325,352}
            \node[endpoint] at (\a:\R) {};
    
        \node at (0,1.98) {$\mathsf{Gen}_{4}=j+1,\qquad \mathtt n(\mathtt M_{j+1})=4$};
        \node[text=red!75!black] at (342:1.85) {$\mathtt M_{j+1}(4)$};
        \node at (0,-1.95)
        {$\mathtt M_{j+1}(4)=S_{x_{j+1}}^{\mathtt{inv}}(\{x_{j+1}\})$};
    \end{scope}
    
    \end{tikzpicture}
        \caption{{Evolution of the marked configuration sequence $\{\mathtt M_j\}_{j \geq 0}$. }}
        \label{fig:marked-configuration-sequence}
    \end{figure}

    Define the tree completion time of $\{ \mathtt M_j \}_{j \geq 0}$ by
    \begin{align}\label{defn:tree completion time}
        T_{\text{tree}} = \min \{ t \geq 1 : \mathtt n(\mathtt M_j) = \mathtt n(\mathtt M_t)\mbox{ for }j \geq t\} ,
    \end{align}
    and the one-arm domination time of $\{ \mathtt M_j \}_{j \geq 0}$ by
    \begin{align}\label{defn:one-arm domination time}
        T_r = \min \{ t \geq 1 : \mathtt C_l(\{ \mathtt M_j \}_{j \geq 0}) = \mathtt C_t(\{ \mathtt M_j \}_{j \geq 0}) \mbox{ for } l\geq t \}.
    \end{align}       
  
    \begin{proposition}\label{prop-equivalence-of-marks-and-chl}
    Suppose $\{ \mathtt M_j\}_{j \geq 0}$ is the marked configuration sequence coupled to a $\text{CHL}_1$ on $\mathbb T_1 \times \mathbb R_+$ with the Poisson point process $\mathcal A = \{ (t_j, x_j): t_j\mbox{ is increasing}\}$. Then we have the following facts:
    \begin{itemize}
    \item[(1)]For any color $a$ that ever occurs in $\{ \mathtt M_l\}_{l \geq 0}$, and every $j \geq \mathsf{Gen}(a)$, there is a unique tree that we denote by $\mathsf{Tree}_a(j)$ that satisfies
    \begin{equation}\label{eq-ref-Mj-trees}
        \mathtt M_j(a)=\mathcal C_{t_j}^{-1}\big(\mathsf{Tree}_a(j)\big)\cap\mathbb T_1\,, \quad \mathsf{Tree}_a(j) \subseteq \mathsf{Tree}_a(j+1)
    \end{equation}
    in the $\text{CHL}_1$ coupled by $\{ \mathtt M_j\}_{j \geq 0}$. In this case, we say the tree $\mathsf{Tree}_a(j)$ is coupled with the interval $\mathtt M_j(a)$ (at time $j$), and the tree $\mathsf{Tree}_a(j)$ has color $a$.
    \item[(2)]If $\mathtt M_j =  [(I_1 ,\ldots , I_k),(1,\ldots , k)]$ and $x_{j+1} \in I_l$ for some $1 \leq l \leq k$, then the particle arriving at $t_{j+1}$ is attached to $\mathsf{Tree}_l(j)$ in the $\text{CHL}_1$ coupled by $\{ \mathtt M_j\}_{j \geq 0}$;
    \item[(3)]$t_{T_{\text{tree}}}=\omega_{N, \lambda}, t_{T_r}=\upsilon_{N, \lambda}$. 
    \end{itemize}
    \end{proposition}
    \begin{proof}
    We prove by a strengthened induction on $j$ that shows (1) and (2) holds for every $j \geq 0$. The case \(j=0\) is vacuous since no non-zero color is present. Assume that it holds at time $t_j$ for some $j\geq 0$. If $x_{j+1}\in\mathtt M_j(l)$ for some non-zero color $l$, then $x_{j+1} \in \mathcal C_{t_j}^{-1}\big(\mathsf{Tree}_l(j)\big)\cap\mathbb T_1$ by our induction hypothesis. Therefore, the new slit at $(t_{j+1},x_{j+1})$ is attached to $\mathsf{Tree}_l(j)$, and $\mathsf{Tree}_l(j)$ is enlarged to $\mathsf{Tree}_l(j) \cup \mathcal C_{t_j}(\bigl[x_{j+1},x_{j+1}+\tfrac{i\lambda}{N}\bigr])$, which is one of the trees of the corresponding $\operatorname{CHL}_1$ at time $t_{j+1}$. Therefore, if we define $\mathsf{Tree}_l(j+1) = \mathsf{Tree}_l(j) \cup \mathcal C_{t_j}(\bigl[x_{j+1},x_{j+1}+\tfrac{i\lambda}{N}\bigr]) \supseteq \mathsf{Tree}_l(j)$, then we have
    \begin{equation*}
        \mathcal C_{t_{j+1}}^{-1}\big(\mathsf{Tree}_l(j+1)\big) \cap \mathbb T_1 ={S_{x_{j+1}}^{-1}\big(\mathtt M_j(l)\cup\bigl[x_{j+1},x_{j+1}+\tfrac{i\lambda}{N}\bigr]\big)
        =S_{x_{j+1}}^{\mathtt{inv}}(\mathtt M_j(l)).}
    \end{equation*}
    also, for any color $a \neq l$ that present at time $t_j$, if we define $\mathsf{Tree}_a(j+1) = \mathsf{Tree}_a(j)$ we have
    \begin{equation*}
        \mathcal C_{t_{j+1}}^{-1}\big(\mathsf{Tree}_a(j+1)\big) \cap \mathbb T_1 =S_{x_{j+1}}^{-1}\big(\mathtt M_j(a)\big)
        =S_{x_{j+1}}^{\mathtt{inv}}(\mathtt M_j(a)).
    \end{equation*}
    Therefore, we always have \eqref{eq-ref-Mj-trees} at time $t_{j+1}$ by the first case of \eqref{formula of Mark configuration sequence}. 
    
    If $x_{j+1}\in\mathtt M_j(0)$, then for each color $a$ that is present at time $t_j$, we have $x_{j+1} \notin \mathcal C_{t_j}^{-1}\big(\mathsf{Tree}_a(j)\big)\cap\mathbb T_1$ by our induction hypothesis. Thus, the new slit at $(t_{j+1},x_{j+1})$ creates a new boundary tree $\mathcal C_{t_j}(\bigl[x_{j+1},x_{j+1}+\tfrac{i\lambda}{N}\bigr])$ with $\mathcal C_{t_j}(x_{j+1}) \in \mathbb T_1$. Suppose $\mathtt n(\mathtt M_j) = k$ and denote this tree by $\mathsf{Tree}_{k+1}(j+1)$. Then we have
    \begin{align*}
    \mathcal C_{t_{j+1}}^{-1}\big(\mathsf{Tree}_{k+1}(j+1)\big) \cap \mathbb T_1 &= S_{x_{j+1}}^{-1}\bigl(\bigl[x_{j+1},x_{j+1}+i\lambda/N\bigr]\bigr) \cap \mathbb T_1=S_{x_{j+1}}^{\mathtt{inv}}(\{x_{j+1}\})\,.
    \end{align*}
    Similarly, for any color $1 \leq a \leq k$, color $a$ is present at time $t_j$. Thus, if we define $\mathsf{Tree}_a(j+1) = \mathsf{Tree}_a(j)$ we have
    \begin{equation*}
        \mathcal C_{t_{j+1}}^{-1}\big(\mathsf{Tree}_a(j+1)\big) \cap \mathbb T_1 =S_{x_{j+1}}^{-1}\big(\mathtt M_j(a)\big)
        =S_{x_{j+1}}^{\mathtt{inv}}(\mathtt M_j(a)).
    \end{equation*}
    Therefore, we always have \eqref{eq-ref-Mj-trees} at time $t_{j+1}$ by the second case of \eqref{formula of Mark configuration sequence}. Thus, we have proved (1) and (2). 
    
    To prove (3), note that the growing color at step $j+1$ is exactly the color of the tree hit at time $t_{j+1}$. Thus the last creation of a non-zero color is $T_{\text{tree}}$, and therefore the last time a new tree is created. Also, $T_r$ is the minimal integer $h$ such that $\{\mathcal C_l(\{\mathtt M_j\}_{j \geq 0})\}_{l \geq 1}$ is eventually a constant color $a_*$ from $l = h$. Since the infinite tree is unique, the infinite tree is $\cup_{j \geq \mathsf{Gen}(a_*)}\mathsf{Tree}_{a_*}(j)$. Therefore, $t_{T_r}$ is the minimal time where the infinite tree grows at time $t_{T_r}$ and no other tree grows after time $t_{T_r}$. Consequently, we have $t_{T_{\text{tree}}}=\omega_{N,\lambda}$ and $t_{T_r}=\upsilon_{N,\lambda}$, which proves (3).
    \end{proof}
 By the above proposition, for $t > t_{T_{\text{tree}}}$ no new tree {is created in the coupled $\text{CHL}_N$ process}, and for $t\geq t_{T_r}$ only one tree will {receive newly attached particles} on the coupled $\text{CHL}_N$. Therefore, we can study the one-arm domination time of $\text{CHL}_N$ by studying the marked configuration sequence generated by an i.i.d.\ sequence $\{ x_j \}$ with $x_1 \sim \mathsf{Unif}(\mathbb T_1)$.
 
Next, given an i.i.d.\ sequence $\{ x_j \}_{j \geq 1}$ with $x_1 \sim \mathsf{Unif}(\mathbb T_1)$ and a marked configuration sequence $\{ \mathtt M_j \}_{j \geq 0}$ generated by $\{ x_j \}_{j \geq 1}$, we shall analyze the expectation of $T_r$. Our proof is based on the analysis of the following procedure on $(\{ \mathtt M_j \}_{j \geq 0} , \{x_j \}_{j \geq 1})$ which yields an upper bound of $T_r$:
\begin{enumerate}[
    label=\textbf{Step \arabic*.},
    leftmargin=3.2em,
    labelwidth=2.6em,
    labelsep=0.6em,
    align=left,
    itemsep=0.4ex,
    topsep=0.6ex,
    parsep=0pt
]
    \item Initialize the whole process by setting the current number of particles $\mathtt{num} \leftarrow 1$, the state-recording variables $\tau \leftarrow 1 ,\mathtt{NUM} \leftarrow 0$, and then we begin the next steps. 
    
    \item Note that
    \begin{align}\label{eq-[x-2arc,x+2arc]}
        S_{x}^{\mathtt{inv}}(\{x\})=\left[x -2\arctan\tfrac{\delta}{\sqrt{1-\delta^2}},  x +2\arctan\tfrac{\delta}{\sqrt{1-\delta^2}}\right]
    \end{align}
    and
    \begin{align*}
        4\arctan\frac{\delta}{\sqrt{1-\delta^2}}>\delta,
    \end{align*}
    i.e. for $l \geq 1$ and $c_l := \mathtt C_l(\{M_j\}_{j \geq 0})$, the tree $\mathsf{Tree}_{c_l}(l)$ is always coupled with an interval longer than $\delta$ in the marked configuration. Therefore, for any $j \geq 1$, we can pick an interval $I$ with $|I|=\delta$ and a color $k\in\mathbb{N}^*$ such that $I\subset \mathtt M_j(k)$.
    
    Thus, we pick a new closed interval $I_*$ with $|I_*|=\delta$ and a color $1 \leq k_* \leq \mathtt n(\mathtt M_{\mathtt{num}})$ such that $I_* \subset \mathtt M_{\mathtt{num}}(k_*)$. Next, we find the minimum $\mathtt{num}'  \geq \mathtt{num} + 1$ such that we have $|S^{\mathtt{inv}}_{x_{\mathtt{num}'}}\circ \cdots \circ S^{\mathtt{inv}}_{x_{\mathtt{num}+1}}(I_*)|\in [0,\delta^3) \cup (2\pi - \delta^3, 2\pi]$. Then, we record the number of attached particles in \textbf{Step 2} by $T'_{\mathtt{NUM}, \tau} \leftarrow\mathtt{num}'  -\mathtt{num}$, where we use this notation since \textbf{Step 2} may be applied multiple times with different $(\mathtt{NUM}, \tau)$-pairs.

    \item Next, we let $\mathtt{num} \leftarrow \mathtt{num}'$ and split into two cases. Set
    \begin{align}
    E_{\mathtt{NUM}, \tau} \leftarrow \{|S^{\mathtt{inv}}_{x_{\mathtt{num}'}}\circ \cdots \circ S^{\mathtt{inv}}_{x_{\mathtt{num}+1}}(I_*)| > 2\pi - \delta^3 \}\,,
    \end{align}
    if $E_{\mathtt{NUM}, \tau}^c$ happens, we set $\tau \leftarrow \tau + 1$ and return to \textbf{Step 2}. Otherwise, $E_{\mathtt{NUM}, \tau}$ happens and in this case we record the number of attached particles in this session by
    \begin{align}\label{hat T, 1}  
        \widehat{{T}}_{\mathtt{NUM}} \leftarrow T_{\mathtt{NUM}, 1} '+ \ldots + T_{\mathtt{NUM}, \tau}'\,,
    \end{align}
    and then we go to \textbf{Step 4}.
    
    \item After \textbf{Step 3}, the unique color $\mathtt m$ such that
    \begin{align*}
        |\mathtt M_{\mathtt{num}} (\mathtt m)|>2\pi -\delta^3
    \end{align*}
    is determined. Define
    \begin{align}\label{eq:def-F-NUM}
    F_{\mathtt{NUM}} = \{\mbox{exists the smallest }\tilde{T}\in N^*\,, x_{\mathtt{num}+\tilde{T}}\in \mathbb{T}_1\setminus \mathtt M_{\mathtt{num}+\tilde{T} -1} (\mathtt m)\}\,.
    \end{align}
     If $F_{\mathtt{NUM}}$ happens, we set $\tilde{T}$ the smallest possible $\tilde{T}$ in the definition of $F_{\mathtt{NUM}}$ and set $\tilde{T}_{\mathtt{NUM}} \leftarrow \tilde{T}$. Then, we set $
    \mathtt{num}\leftarrow \mathtt{num}+\tilde{T}$, $\mathtt{NUM} \leftarrow \mathtt{NUM} + 1$, $\tau \leftarrow 1$ and return to \textbf{Step 2}. Otherwise, for every $h\in\mathbb N^*$,
    \[
        x_{\mathtt{num}+h}\in \mathtt M_{\mathtt{num}+h-1}(\mathtt m).
    \]
    Thus all subsequent particles are attached to the single color $\mathtt m$, which yields {a} one-arm domination. We then terminate the process and record $\mathtt{NUM}\,, \{\widehat{{T}}_{k}\}_{0 \leq k \leq \mathtt{NUM}}$ and $\{\tilde{{T}}_{k}\}_{0 \leq k \leq \mathtt{NUM}-1}$.
\end{enumerate}

By our construction of $\mathtt{NUM}$, $\{\widehat{{T}}_{k}\}_{0 \leq k \leq \mathtt{NUM} }$ and $\{\tilde{{T}}_{k}\}_{0 \leq k \leq \mathtt{NUM} - 1}$ in the above steps, the one-arm domination time $T_r$ of $\{ \mathtt M_j\}_{j\geq 0}$ is bounded by
    \begin{align}\label{eq:one-arm domination time}
        T_r\le 1+\sum_{k=0}^\mathtt{NUM}\widehat{{T}}_{k}+\sum_{k=0}^{\mathtt{NUM}-1}\tilde{{T}}_{k},
    \end{align}
 where $\mathtt{NUM}$, $\{\widehat{{T}}_{k}\}_{0 \leq k \leq \mathtt{NUM} }$ and $\{\tilde{{T}}_{k}\}_{0 \leq k \leq \mathtt{NUM} - 1}$ are a.s. finite random variables determined by $\{ x_k \sim \mathsf{Unif}(\mathbb T_1) \,,\, x_k\ \mbox{i.i.d.}\}_{k \geq 1}$. {It remains to prove} $\E[T_r]\le  C'_4\delta^{-3}$. We first bound the expectation of $\widehat{{T}}_{0}$.

\begin{lemma}\label{lem-bound-T1-prime}
There exists a constant $C'_3$ such that $\E[\widehat{{T}}_{0}] \leq C'_3 \delta^{-3}$. 
\end{lemma}
\begin{proof}
{By construction and the strong Markov property, each interval exploration in \textbf{Step 2} starts from an interval of length $\delta$, and the future driving points are  i.i.d. uniform on $\mathbb T_1$. Hence the pairs $(T'_{0,i},\mathbf 1_{E_{0,i}})_{i\ge1}$ are i.i.d.. The second estimate in Lemma~\ref{lem-harmonic-estimations-quadratic} gives}
\begin{align}\label{eq-expectation-T-prime-repaired}
    \mathbb{E}[T'_{0,1}]\le C_{*}\delta^{-2} \mbox{ for some constant }C_*\,.
\end{align}
{Next, note that for any $\{J_j\}_{j\ge0}\sim\mathcal P_{J_0}$ with $|J_0|=\delta$, set $\sigma_0=\inf\{k:|J_k|<\delta^3\ \text{or}\ |J_k|>2\pi-\delta^3\}$ and $E=\{|J_{\sigma_0}|>2\pi-\delta^3\}$. Since $|J_k|$ is a bounded martingale, optional stopping theorem gives}
\begin{align*}
    {\delta=\mathbb{E}[|J_{\sigma_0}|]
    =\mathbb{P}(E)\mathbb{E}[|J_{\sigma_0}|\mid E]+(1-\mathbb{P}(E))\mathbb{E}[|J_{\sigma_0}|\mid E^c].}
\end{align*}
{Using $2\pi-\delta^3\le |J_{\sigma_0}|\le2\pi$ on $E$ and $0\le |J_{\sigma_0}|\le\delta^3$ on $E^c$, for all sufficiently small $\delta$ we have}
\begin{align*}
    c\delta\le \P(E)\le C\delta \mbox{ for some constants }c, C\,,
\end{align*}
therefore, by the construction of $\{E_{0,i}\}_{i \geq 1}$ in our procedure, we have 
\begin{align}\label{eq-success-prob-repaired}
    c\delta\le \mathbb{P}(E_{0,i})\le C\delta \mbox{ for all }i \geq 1\,,
\end{align}
{Thus $\tau=\inf\{i\ge1:E_i\ \mathrm{occurs}\}$ is stochastically dominated by $\mathsf{Geometric}(c\delta)$, and}
\begin{align*}
    {\mathbb{E}[\widehat T_0]=\mathbb{E}\left[\sum_{i=1}^{\tau}T'_{0,i}\right]
    =\sum_{i\ge1}\mathbb{E}\left[T'_{0,i}\mathbf 1_{\{\tau\ge i\}}\right]=\sum_{i\ge1}\mathbb{P}(\tau\ge i)\mathbb{E}[T'_{0,1}]
    =\frac{\mathbb{E}[T'_{0,1}]}{c\delta}\le C'_3\delta^{-3},}
\end{align*}
{where we used that $\{\tau\ge i\}$ depends only on the first $i-1$ interval explorations and is independent of $T'_{0,i}$.}
\end{proof}
{Now we are ready to prove $\E[T_r]\le  C'_4\delta^{-3}$.} 
\begin{lemma}\label{lem-bound-Tr}
There exists a constant $C'_4$ such that $\E[T_r] \leq C'_4 \delta^{-3}$. 
\end{lemma}
\begin{proof}
    Recall \textbf{Step 1} to \textbf{Step 4}. Once our procedure enters \textbf{Step 4}, the variable \(\mathtt{NUM}\) is added by $1$ if and only if $F_{\mathtt{NUM}}$ happens, which has probability bounded by $1-(1-4\delta^{-1}\cdot \delta^{3}) = 4\delta^2$ by Item (6) of Lemma~\ref{lem-properties-mathcal-P}. Also, by strong Markov property we have $\{\mathbf 1_{F_i}\}_{i \geq 1}$ are i.i.d. random variables, which gives
\begin{align}\label{eq:X,t}
\P(\mathtt{NUM}\ge k)\leq (4\delta^{2})^k \,,\quad k\ge 0\,.
\end{align}
Hence, $\mathbb{E}[\mathtt{NUM}]\le 2$ for all sufficiently small $\delta$. Moreover, after a successful domination attempt, the exceptional complement has length at most \(\delta^3\). Moreover, if we write $\tilde{T} = +\infty$ when $\tilde{T}$ does not exist in \textbf{Step 4}, then for any $k \geq 0$ we have $|\mathtt M_{\mathtt{num} + k}(\mathtt m)| \geq 2\pi- (1-\frac{1}{2}\delta^2)^k \delta^3$ conditioned on $\tilde{T} > k$ by Item (1) of Lemma~\ref{lem-properties-mathcal-P}. Therefore, we have that for $k\geq 1$
\begin{equation}
\P(\tilde{T}_{j} = k \mid \tilde{T}_{j} > k - 1)\leq \tfrac{1}{2\pi}\bigl(1-\tfrac{1}{2}\delta^2\bigr)^{k-1} \delta^{3}\,, 
\end{equation}
which gives
\begin{align}\label{eq:tilde-T-i-upper-bound}
\P(\tilde{T}_{j} = k) = \P(\tilde{T}_{j} = k \,, \tilde{T}_{j} > k - 1) \leq \P(\tilde{T}_{j} = k \mid \tilde{T}_{j} > k - 1)\leq \tfrac{1}{2\pi}\bigl(1-\tfrac{1}{2}\delta^2\bigr)^{k-1} \delta^{3}.
\end{align}
Consequently, we have
\begin{align}\label{expectation-temp,T0}
    {\mathbb{E}[\mathtt{NUM}]\leq \sum_{k \geq 1}(4\delta^2)^k < 1 \,, \quad \mathbb{E}[\tilde{T}_{j}\mathbf 1_{\{\tilde{T}_{j} < +\infty\}}]\le \sum_{k \geq 1}\frac{k\bigl(1-\tfrac{1}{2}\delta^2\bigr)^{k-1}\delta^3}{2\pi} =\frac{2}{\pi\delta}.}
\end{align}

{Note that $\{\widehat{{T}}_{k}\}_{0 \leq k \leq \mathtt{NUM} }$ and $\{\tilde{{T}}_{k}\mathbf 1_{\{\tilde{T}_k < +\infty\}}\}_{0 \leq k \leq \mathtt{NUM}}$ are i.i.d. random variables, and $\mathtt{NUM}$ is an explosion time adapted to $\{\tilde{{T}}_{k}\mathbf 1_{\{\tilde{T}_k < +\infty\}}\}_{k \geq 0}$. Therefore, combining \eqref{expectation-temp,T0} and Lemma~\ref{lem-bound-T1-prime} with \eqref{eq:one-arm domination time}, there exists a constant $C'_4>0$ such that}
\begin{align*}
    \mathbb{E}[T_r]&\leq 1+\mathbb{E}\left[\sum_{k=0}^\mathtt{NUM}\widehat{{T}}_{k}\right]+\mathbb{E}\left[\sum_{k=0}^{\mathtt{NUM}}\tilde{T}_{k} \mathbf 1_{\{\tilde{T}_{k}< +\infty\}}\right]\le 1+(1+\mathbb{E}[\mathtt{NUM}])(\mathbb{E}[\widehat{T}_0]+\mathbb E[\tilde{T}_0\mathbf 1_{\{ \tilde{T}_0 < + \infty\}}])\\
    &\le C'_4\delta^{-3}\,,
\end{align*}
which completes the proof.
\end{proof}
Now we return to the upper bound of $\E[\upsilon_{N, \lambda}]$ in Theorem~\ref{thm: one-arm domination time}. Note that we have
    \begin{align*}
        \E[\upsilon_{N, \lambda}]=\E[t_{T_r}]=\frac{\E[T_r]}{2\pi N}\le C\lambda^{-3}N^2\,,
    \end{align*}
    where the first equality holds by Item (3) of Proposition~\ref{prop-equivalence-of-marks-and-chl}, the second equality holds by \eqref{eq:t_T and T}, and the first inequality holds by Lemma~\ref{lem-bound-Tr} and $\delta = O(\frac{\lambda}{N})$. Therefore, we obtain the upper bound of $\mathbb{E}[\upsilon_{N, \lambda}]$ in Theorem \ref{thm: one-arm domination time}.

\section{Tail of one-arm domination time}\label{sec:Tail of one-arm domination time}
{This section strengthens the expectation result to the exponential tail bound in Theorem \ref{thm:tail of one-arm domination time}. The lower tail follows from the fact that a semicircular interval can remain in an interior range, while the upper tail iterates the marked configuration mechanism. The strong Markov property converts these estimates into exponential bounds.}

\begin{proof}[Proof of Theorem \ref{thm:tail of one-arm domination time}] We first prove the  lower bound.  
Consider $I^{(\theta)}=[0,\theta), \frac{2}{3}\pi\le \theta\le \frac{4}{3}\pi$. Given the Poisson point process $\{(x_i,t_i) : t_1 < t_2 < \ldots\}$. Denote $I^{(\theta)}_{k} := \chl_{t_k}^{\mathtt{inv}}(I^{(\theta)})$. 
Define 
\begin{align*}
    \eta^{(\theta)}&:=\inf\left\{k\in\mathbb{N}^*: |I^{(\theta)}_{k}|<\theta-\frac{\pi}{3}\ \text{or}\ |I^{(\theta)}_{k}|>\theta+\frac{\pi}{3}\right\},\\
    \sigma'^{(\theta)}&:=\inf\left\{k\in \mathbb{N}^*: |I^{(\theta)}_{k}|<\frac{\pi}{3}\ \text{or}\ |I^{(\theta)}_{k}|>\frac{5\pi}{3}\right\}.
\end{align*}
{Set $n_\delta:=\left\lfloor(1000\delta^3)^{-1}\right\rfloor$.} For sufficiently large $N$, with Corollary \ref{cor:lem-harmonic-estimations-quadratic},
\begin{align*}
    {\frac{\pi^3}{97\delta^3}\le \mathbb{E}[\eta^{(\theta)}]\le \frac{\pi^3}{95\delta^3}.}
\end{align*}
{Moreover, by the strong Markov property and the upper bound in Corollary \ref{cor:lem-harmonic-estimations-quadratic},}
\begin{align*}
    {\mathbb E\big[\eta^{(\theta)}\mid \eta^{(\theta)}>n_\delta\big]\le n_\delta+\frac{\pi^3}{95\delta^3}.}
\end{align*}
So
\begin{align*}
    {\frac{\pi^3}{97\delta^3}}\le\mathbb{E}[\eta^{(\theta)}]=&\mathbb{E}
    \left[\eta^{(\theta)}\mid \eta^{(\theta)}> n_\delta\right]\mathbb{P}\left(\eta^{(\theta)}> n_\delta\right)+\mathbb{E}
    \left[\eta^{(\theta)}\mid \eta^{(\theta)}\le n_\delta\right]\mathbb{P}\left(\eta^{(\theta)}\le n_\delta\right)\\
    \le& \left( n_\delta+{\frac{\pi^3}{95\delta^3}}\right)\mathbb{P}\left(\eta^{(\theta)}> n_\delta\right)+ n_\delta\mathbb{P}\left(\eta^{(\theta)}\le n_\delta\right),
\end{align*}
which implies that 
\begin{align}\label{lower bound of hitting time}
    \mathbb{P}\left(\eta^{(\theta)}> n_\delta\right)\ge 0.9.
\end{align}
{We first prove the following proposition.}
\begin{proposition}\label{prop: interval greater than average}
    For any $k\in \mathbb{N}^*$, 
    \begin{align*}
        \mathbb{P}\left(|I^{\left(\frac{2}{3}\pi\right)}_k|\ge \frac{2}{3}\pi\right)\ge \frac{1}{3}.
    \end{align*}
\end{proposition}
\begin{proof}
    Define 
    \begin{align*}
        \widehat{I}^j:=\frac{2(j-1)}{3}\pi +[0,\frac{2}{3}\pi), j=1,2,3,\qquad \widehat{I}^j_{k}:= \chl_{t_k}^{\mathtt{inv}}(\widehat{I}^j).
    \end{align*}
    {The three intervals $\widehat I^1,\widehat I^2,\widehat I^3$ partition $\mathbb T_1$ up to endpoints. Hence, for any fixed $k\in\mathbb N$,}
    \begin{align*} 
    {\sum_{j=1}^3|\widehat{I}^j_{k}|=2\pi.}
    \end{align*}
    which implies that 
    \begin{align*}
{\sum_{j=1}^3\mathbf{1}_{\{|\widehat{I}^j_{k}|\ge\frac{2}{3}\pi\}}\ge1.}
    \end{align*}
    Since $\widehat{I}^j_{k}$ {has the same distribution for different $j$, this completes the proof}.
\end{proof}
With (\ref{lower bound of hitting time}) and Proposition \ref{prop: interval greater than average}, we obtain that for any $\frac{2}{3}\pi\le\theta\le\frac{4}{3}\pi$, if $\frac{2}{3}\pi\le \theta \leq \pi$, then
\begin{align}
    {\mathbb{P}\left(\sigma'^{(\theta)}>n_\delta, |I^{(\theta)}_{n_\delta}|\in\left[\frac{2}{3}\pi, \frac{4}{3}\pi\right]\right)}&\ge {\mathbb{P}\left(\eta^{(\theta)}> n_\delta,|I^{\left(\frac{2}{3}\pi\right)}_{n_\delta}|\ge \frac{2}{3}\pi \right)}\nonumber \\
    &\ge\mathbb{P}\left(|I^{\left(\frac{2}{3}\pi\right)}_{n_\delta}|\ge \frac{2}{3}\pi \right)-\mathbb{P}\left(\eta^{(\theta)}\le n_\delta\right)\nonumber \\
    &\ge \frac{1}{3}-0.1\ge 0.2.\label{eq-case-smaller-than-pi}
\end{align}
 If $\pi\le \theta \leq \frac{4}{3}\pi$, then {by  symmetry}
\begin{align*}
    {\mathbb{P}\left(\sigma'^{(\theta)}>n_\delta, |I^{(\theta)}_{n_\delta}|\in\left[\frac{2}{3}\pi, \frac{4}{3}\pi\right]\right)\geq 0.2.}
\end{align*}
Therefore, in either case we have 
\begin{equation}\label{eq:block-lower-interior}
    {\mathbb{P}\left(\sigma'^{(\theta)}>n_\delta, |I^{(\theta)}_{n_\delta}|\in\left[\frac{2}{3}\pi, \frac{4}{3}\pi\right]\right) \geq 0.2\,,\qquad \frac{2\pi}{3}\le\theta\le\frac{4\pi}{3}.}
\end{equation}
{The estimate \eqref{eq:block-lower-interior} is uniform over the interior range. Starting from any interval whose length lies in $[2\pi/3,4\pi/3]$, the strong Markov property gives}
\begin{align*}
    {\mathbb P\left(\sigma'^{(\pi)}>m n_\delta\right)\ge (0.2)^m,\qquad m\ge1,}
\end{align*}
{Taking $m=\left\lceil \lambda^{-3}N^3x/n_\delta\right\rceil$ and using \eqref{eq-def-delta}, there exist $\widehat C_1,\widehat C_2>0$ such that}
\begin{align*}
    \mathbb{P}(\sigma'^{(\pi)}>\lambda^{-3}N^3 x)\ge \widehat{C}_1\mathrm{e}^{-\widehat{C}_2 x}.
\end{align*}
Review the definition of $\sigma'^{(\pi)}$ and $A^{(\sigma'^{(\pi)})}$ in (\ref{stopping time-sigma'}) and (\ref{event A^{(sigma')}}). {The conditional lower bound for $A^{(\sigma'^{(\pi)})}$ appearing in the proof of Lemma~\ref{lem-lower-bound-ET} is uniform at the stopping time. Therefore, by the strong Markov property,}
\begin{align*}
    \mathbb{P}(\sigma'^{(\pi)}\mathbf{1}_{A^{(\sigma'^{(\pi)})}}>\lambda^{-3}N^3 x)&=\mathbb{P}(\sigma'^{(\pi)}>\lambda^{-3}N^3 x, A^{(\sigma'^{(\pi)})})\\
    &\ge \mathbb{P}(\sigma'^{(\pi)}>\lambda^{-3}N^3 x)\times \frac{1}{7}\times\frac{1}{2}\ge \frac{1}{14}\widehat{C}_1\mathrm{e}^{-\widehat{C}_2 x}.
\end{align*}
Back to the Poisson point process $\{(x_i,t_i) : t_1 < t_2 < \ldots\} $ in the unrescaled $\operatorname{CHL}_N$, where we know that $\upsilon_{N, \lambda}=t_{T_r}$. {We use the following standard Poisson estimates for the arrival times:}
\begin{equation}\label{eq:estimate of t_k}
{
\begin{aligned}
    \mathbb P(t_k>u)\ge c_0, \ \text{if } \frac{k}{2\pi N}\ge 2u;\quad     \mathbb P(t_k>u)\le \exp(-c_1 2\pi Nu),\ \text{if } 2\pi Nu\ge 2k,
\end{aligned}}
\end{equation}
{where $c_0,c_1>0$ are universal constants. Indeed, the first bound follows from Paley--Zygmund inequality, while the second follows from Lemma \ref{lem:Chernoff Bound}.}
{Take $x=4\pi t$ and $k_0=\lceil \lambda^{-3}N^3x\rceil$. Then $\mathbb E[t_{k_0}]=k_0/(2\pi N)\ge 2\lambda^{-3}N^2t$. Hence the first bound in \eqref{eq:estimate of t_k} gives $\mathbb P(t_{k_0}>\lambda^{-3}N^2t)\ge c_0$.} Note that $\upsilon_{N, \lambda}\ge t_{{\sigma'^{(\pi)}} \mathbf{1}_{A^{(\sigma'^{(\pi)})}}}$, {so for any $t>0$}
\begin{align*}
    {\mathbb{P}(\upsilon_{N, \lambda}>\lambda^{-3} N^2t)}&{\ge\mathbb{P}(t_{k_0}>\lambda^{-3} N^2t, {\sigma'^{(\pi)}} \mathbf{1}_{A^{(\sigma'^{(\pi)})}}>k_0)}\\
    &{=\mathbb{P}(t_{k_0}>\lambda^{-3} N^2t)\mathbb{P}({\sigma'^{(\pi)}} \mathbf{1}_{A^{(\sigma'^{(\pi)})}}>k_0)}
    {\ge c_0\,\frac{1}{14}\widehat{C}_1\mathrm{e}^{-4\pi\widehat{C}_2 t}\ge C_3\mathrm{e}^{-C_4t}.}
\end{align*}
{This completes the proof of the lower bound.}

Now we consider the upper bound. We continue to use the same notation and review \textbf{Step 1} to \textbf{Step 4} in Subsection \ref{upper bound}.  Define the interval $J^{(\theta)} \subset \mathbb{T}_1$ and $|J^{(\theta)}|=\theta, \delta^3\le \theta\le 2\pi-\delta^3$. Given the Poisson point process $\{(x_j,t_j) : t_1 < t_2 < \ldots\}$. Denote $J^{(\theta)}_{k} := (\chl_{t_k})^{\mathtt{inv}}(J^{(\theta)})$. 
Define 
\begin{align*}
    \sigma^{(\theta)}:=\inf\left\{k\in \mathbb{N}^*: |J^{(\theta)}_{k}|<\delta^3\ \text{or}\ |J^{(\theta)}_{k}|>2\pi-\delta^3\right\}.
\end{align*}
\begin{lemma}\label{lem:tail-sigma-theta}
{There exists a constant $C'_5$ such that, uniformly for $\delta^3\le\theta\le2\pi-\delta^3$,}
\begin{align*}
    {\mathbb E[\sigma^{(\theta)}]\le C'_5\delta^{-3}.}
\end{align*}
\end{lemma}
\begin{proof}
{Let $D_k:=|J_k^{(\theta)}|\wedge(2\pi-|J_k^{(\theta)}|)$.  Applying the proof of Lemma \ref{lem-harmonic-estimations-quadratic} and the lower bounds in Lemma \ref{estimation of I(N-I)-1} gives}
\begin{align*}
    {\sup_{\delta/10\le D_0\le \pi}\mathbb E\big[\inf\{k:D_k\le\delta/10\}\big]\le C\delta^{-3}.}
\end{align*}
{Once $D_k\le\delta/10$, we apply Lemma \ref{lem-properties-mathcal-P} to the shorter interval, or to its complement if $2\pi-|J_k^{(\theta)}|\le\delta/10$. With probability at least $1/2$ this shorter interval decreases to length at most $\delta^3$, and each attempt has expected duration at most $C\delta^{-2}$; if the attempt fails, Lemma \ref{lem-properties-mathcal-P} returns it to length $O(\delta)$. The same strong Markov property used in Lemma \ref{lem-harmonic-estimations-quadratic} therefore gives an additional expected time at most $C\delta^{-2}$. Combining the two bounds proves the claim.}
\end{proof}
We review the definition of $T_r, \widehat{{T}}_{j}$ and $\tilde{T}_j$ in (\ref{defn:one-arm domination time}), (\ref{hat T, 1})  and \textbf{Step 4} in Subsection \ref{upper bound}. With Lemma~\ref{lem-bound-T1-prime}, 
\begin{align}\label{T-hat}
    \mathbb{E}[\widehat{{T}}_{j}]\le C'_3\delta^{-3}.
\end{align}
{Note that $\widehat{{T}}_{j}$ are i.i.d. random variables. Within one run of the above procedure, if the event $\{\widehat T_j>(r-1)s\}$ occurs, then at time $(r-1)s$ the process is either in the middle of one interval exploration or at the beginning of the next one. By the strong Markov property, the remaining expected time is bounded by the sum of the uniform bound in Lemma \ref{lem:tail-sigma-theta} and \eqref{T-hat}. So we obtain that for $r\ge1$}
\begin{align*}
    \mathbb{P}(\widehat{{T}}_{j}>rs\mid\widehat{{T}}_{j}>(r-1)s)&\le \max_{\delta^3\le \theta\le 2\pi-\delta^3}\mathbb{P}(\sigma^{(\theta)}+\widehat{{T}}_{j}>s)\\
    &\le\max_{\delta^3\le \theta\le 2\pi-\delta^3}\frac{\mathbb{E}[\widehat{{T}}_{j}]+\mathbb{E}[\sigma^{(\theta)}]}{s}
    \le \frac{\widehat{C}}{s\delta^{3}}, 
\end{align*}
which implies that
\begin{align*}
\mathbb{P}\left(\widehat{T}_j>ks\right)=\mathbb{P}(\widehat{{T}}_{j}>s)\prod_{r=2}^k\mathbb{P}(\widehat{{T}}_{j}>rs\mid\widehat{{T}}_{j}>(r-1)s)\le \left(\frac{\widehat{C}}{s\delta^{3}}\right)^k.
\end{align*}
Take $s=\widehat{C}\delta^{-3} \mathrm{e}$, we get that there exists $K>0$ such that $\mathbb{P}(\widehat{T}_j>\delta^{-3}x)\le \text{e}^{-Kx}.$ 
For any $x>0$, with (\ref{eq:X,t}) and Lemma \ref{lem:Chernoff Bound}, 
\begin{align}
\mathbb{P}\left(\sum_{j=0}^{\mathtt{NUM}}\widehat{T}_{j}>\delta^{-3}x\right)=&\mathbb{E}\left[\mathbb{P}\left(\sum_{j=0}^{\mathtt{NUM}}\widehat{T}_{j}>\delta^{-3}x\right) \Big | \mathtt{NUM}\right]\nonumber\\
\le&\sum_{k=0}^{[x]}\mathbb{P}(\mathtt{NUM}=k)\mathbb{P}\left(\sum_{j=0}^{k}\widehat{T}_j>\delta^{-3}x\right)+\mathbb{P}(\mathtt{NUM}>[x]+1)\nonumber\\
\le& \sum_{k=0}^{[x]}100\delta^{2k}\left( \dfrac{ K\text{e}x}{k+1} \right)^{k+1} \text{e}^{-Kx}+200\delta^{2x}\nonumber\\
\le& 100\text{e}^{-Kx}\sum_{k=0}^{[x]}\left( \dfrac{K\text{e}x \delta^{2}}{k+1 } \right)^{k+1}+200\delta^{2x}\nonumber\\
\le& 100\text{e}^{-Kx}\left(\sum_{k=0}^{[K\text{e}x\delta^{2}]}\left( \dfrac{K\text{e}x \delta^{2}}{k+1} \right)^{k+1}+[x]\right)+200\delta^{2x}\nonumber\\
\le& 100\text{e}^{-Kx}\left(\left( Kx e \delta^{2} \right)^{[Kxe\delta^2]+2}+[x]\right)+200\delta^{2x}
\le C_9\text{e}^{-\frac{Kx}{2}}. \label{eq:sum hat T}
\end{align}
Note that $\delta^{x}$ decays faster than $e^{-x}$. Based on the proof of (\ref{eq:sum hat T}) and with (\ref{eq:X,t}), 
\begin{align}
{\mathbb{P}\left(\sum_{j=0}^{\mathtt{NUM}-1}\tilde{T}_{j}>\delta^{-3}x\right)\le C_{10}\text{e}^{-C_{11}x}.} \label{eq:sum t}
\end{align}

Combining \eqref{eq-def-delta} ,  \eqref{eq:sum hat T} and \eqref{eq:sum t} with \eqref{eq:one-arm domination time} we derive that there exists $C_{12}, C_{13} >0$ such that
\begin{align}\label{eq:estimate of T_r}
    \mathbb{P}\left(T_r>\lambda^{-3}N^3x\right)\le C_{12}\text{e}^{-C_{13}x}.
\end{align}
 Now we take $x=\pi t$. With \eqref{eq:estimate of T_r} and \eqref{eq:estimate of t_k}, we obtain that for any $t>0$ 
\begin{align*}
    &\,\,\quad\mathbb{P}(\upsilon_{N, \lambda}>1000\lambda^{-3}N^2t)=\mathbb{P}(t_{T_r}>1000\lambda^{-3}N^2t)\\
    &\le \mathbb{P}(t_{T_r}>1000\lambda^{-3}N^2t, T_r>\pi\lambda^{-3}N^3t)
    +\mathbb{P}(t_{T_r}>1000\lambda^{-3}N^2t, T_r\le \pi\lambda^{-3}N^3t)\\
    &\le \mathbb{P}(T_r>\pi \lambda^{-3}N^3t)+\mathbb{P}(t_{\lceil \pi\lambda^{-3}N^3t\rceil}>1000\lambda^{-3}N^2t)\\
    &{\le C_{12}\text{e}^{-C_{13}\pi t}+C_{14}\text{e}^{-C_{15}\lambda^{-3}N^3t}\le C_{1}\text{e}^{-C_2 t},} 
\end{align*}
which completes the proof of  the upper bound.

\end{proof}


\section{{An upper bound for the expected tree completion time}}\label{sec:Expectation of tree completion time}
{In this section we prove Theorem \ref{thm: tree completion time} by using the marked configuration introduced in Section \ref{sec:Expectation of one-arm domination time} to track the zero-colored set, namely the boundary intervals from which no tree has yet grown, and  obtain an upper bound for the expected tree completion time.}
\begin{proof}[Proof of Theorem \ref{thm: tree completion time}]
Given the Poisson point process $\{(t_j, x_j) : t_1 < t_2 < \ldots\}$. In this section, we use the same notation in {Subsection} \ref{upper bound}. {Recalling} Definition \ref{defn:Mark configuration} and \ref{defn:Mark configuration sequence coupled}, for marked configuration sequence $\{ \mathtt M_j \}_{j \geq 0}$ coupled to the $\text{CHL}_N$, the number of colors of $\mathtt M_j$ is $\mathtt n(\mathtt M_j)$, and the tree completion time of $\{ \mathtt M_j \}_{j \geq 0}$ is 
    \begin{align*}
        T_{\text{tree}} = \min \{ t \geq 1 : \mathtt n(\mathtt M_j) = \mathtt n(\mathtt M_t)\mbox{ for }j \geq t\} .
    \end{align*} From Definition \ref{def:tree} and Proposition \ref{prop-equivalence-of-marks-and-chl},  we note that $\omega_{N,\lambda}=t_{T_{\text{tree}}}$. 
    We recall that $\mathtt M_j =[(I_1 , I_2 , \ldots , I_k), (1 , \ldots , k)]$, where $k=\mathtt n(\mathtt M_j)$, and $\mathtt M_j(0)=I_0=\overline{\mathbb T_1 \setminus (\bigcup_{l=1}^{k} I_l)}$, is a union of disjoint closed intervals $J_1 , \ldots, J_{{\ell_j}}$ and we denote
    \begin{align*}
        \mathtt M_j(0):=\bigsqcup_{m=1}^{{\ell_j}} J_m.
    \end{align*}
    So $|\mathtt M_j(0)|=\sum_{m=1}^{{\ell_j}}|J_m|.$ Now we analyze $|\mathtt M_j(0)|$ to describe the total length of non-tree intervals. 
    For any interval $J \subset \mathbb T_1$, we define $\hat{S}_x^{\mathtt{inv}}(J)$ as 
\begin{align*}
     \hat{S}_x^{\mathtt{inv}}(J) =
    \begin{cases}
     \overline{S_x^{\mathtt{inv}}(J \setminus \{x \})},\quad  &  x \in J\,,  \\
    S_x^{\mathtt{inv}}(J), & x \notin J\,,
    \end{cases}
\end{align*}
where this specific definition is inspired {by} the fact that the interval created by a newly grown slit on $\mathbb T_1$ is no longer colored by $0$ under our setting of {marked} configurations. {Analogously to (\ref{formula of Mark configuration sequence}), we obtain} $\{ \mathtt M_j(0) \}_{j \geq 0}$ recursively as follows.
\begin{enumerate}
   \item $\mathtt M_0(0) = \mathbb T_1$;
    \item $\mathtt M_{j+1}(0) = \hat{S}_{x_{j+1}}^{\mathtt{inv}} (\mathtt M_{j}(0))$, where for $\mathtt M_{j}(0) = \bigsqcup_{m=1}^{{\ell_j}} J_m$ we have (denote $a_\delta := 2\arctan{\tfrac{\delta}{\sqrt{1-\delta^2}}}$)
     \begin{align}
     &\quad\,\,\hat{S}_{{x_{j+1}}}^{\mathtt{inv}} (\mathtt M_{j}(0)) =\bigsqcup_{m=1}^{{\ell_j}}\hat{S}_{{x_{j+1}}}^{\mathtt{inv}}(J_m)\nonumber \\
     &=\begin{cases}
     \Big(\bigsqcup_{\substack{1\le m\le {\ell_j} \\  m\neq l}}S_{x_{j+1}}^{\mathtt{inv}}(J_m)\Big)\bigsqcup [S_{x_{j+1}}^{\mathtt{inv}}(a), x_{j+1}-a_\delta ]\bigsqcup [x_{j+1}+a_\delta,S_{x_{j+1}}^{\mathtt{inv}}(b)] \,, \\
    \hfill \exists 1\le l\le {\ell_j},  \ x_{j+1} \in {J_l}:=[a,b]\,,  \\
    \bigsqcup_{m=1}^{{\ell_j}} S_{x_{j+1}}^{\mathtt{inv}}(J_m),  \qquad\qquad\qquad\qquad\quad \mbox{otherwise.}
    \end{cases}
    \label{formula of Mark configuration sequence (0)}
    \end{align}
    by our definitions of $S_{x}^{\mathtt{inv}}(A)$ and $\hat{S}_{x}^{\mathtt{inv}}(A)$ for $A \subset \mathbb T_1$ and $x \in \mathbb T_1$.
\end{enumerate} 
Now we analyze the variation of each $J_m$ when a slit {attaches} to $\mathbb T_1$. 
\begin{lemma}\label{lem: recursion of inverse of tilde-S}
    For any interval $J \subset \mathbb T_1$, we obtain the expectation of $|\hat{S}_x^{\mathtt{inv}}(J)|$ as follows.
    \begin{align*}
        \mathbb{E}\big[|\hat{S}_x^{\mathtt{inv}}(J)| \big| J\big]=\left(1-\frac{2\arctan\frac{\delta}{\sqrt{1-\delta^2}}}{\pi}\right) |J|.
    \end{align*}
\end{lemma}
\begin{proof}
     Assume that $J=[-a,a]$. Note that the slit position $x$ {is uniformly distributed} on $\mathbb{T}_1$. For $0<a\le \frac{\pi}{2}$, with (\ref{eq-S0-1,delta-inverse-explicit}) we have that
    \begin{align*}
        \mathbb{E}[|\hat{S}_x^{\mathtt{inv}}(J)| | J=[-a,a]]&=\frac{1}{\pi}\int_{0}^{\pi} |\hat{S}_x^{\mathtt{inv}}([-a, a])| dx\\
        &=\frac{1}{\pi}\int_0^a\bigl(S_0^{\mathtt{inv}}(\{a+x\})-2\arctan{\tfrac{\delta}{\sqrt{1-\delta^2}}}\bigr)\,dx\\
        &+\frac{1}{\pi}\int_0^a\bigl(S_0^{\mathtt{inv}}(\{a-x\})-2\arctan{\tfrac{\delta}{\sqrt{1-\delta^2}}}\bigr)\,dx\\
        &+\frac{1}{\pi}\int_a^{\pi-a}\bigl(S_0^{\mathtt{inv}}(\{a+x\})-S_0^{\mathtt{inv}}(\{a-x\})\bigr)\,dx\\
        &+\frac{1}{\pi}\int_{\pi-a}^{\pi}\bigl(\pi-S_0^{\mathtt{inv}}(\{x-a\})+\pi-S_0^{\mathtt{inv}}(\{2\pi-x-a\})\bigr)\,dx\\
        &=\frac{2a}{\pi}\bigl(\pi-2\arctan{\tfrac{\delta}{\sqrt{1-\delta^2}}}\bigr).
    \end{align*}  
{The case $\frac{\pi}{2}<a\le \pi$ is similar.}  Therefore, we get the argument.
\end{proof}

{Combining} Lemma \ref{lem: recursion of inverse of tilde-S} with (\ref{formula of Mark configuration sequence (0)}), we have that
\begin{align*}
    \mathbb{E}[|\mathtt M_{j+1}(0)|||\mathtt M_j(0)|]=\left(1-\frac{2\arctan\frac{\delta}{\sqrt{1-\delta^2}}}{\pi}\right)|\mathtt M_j(0)|, 
\end{align*}
which implies that
\begin{align*}
    \mathbb{E}[|\mathtt M_k(0)|]=2\pi \left(1-\frac{2\arctan\frac{\delta}{\sqrt{1-\delta^2}}}{\pi}\right)^k.
\end{align*}
For any $\varepsilon>0$, {set $m_N:=\left\lceil(1+\varepsilon)\frac{\pi}{\lambda}N\log N\right\rceil$. Since a new color is created at step $k+1$ precisely when $x_{k+1}\in\mathtt M_k(0)$,} we have
\begin{align*}
    {\mathbb{P}\left(T_{\text{tree}}\ge m_N\right)}
    &{\le\sum_{k=m_N-1}^{\infty}\mathbb{P}(x_{k+1}\in \mathtt M_k(0))
    =\sum_{k=m_N-1}^{\infty}\frac{\mathbb{E}[|\mathtt M_k(0)|]}{2\pi}}\\
    &{\le\sum_{k=m_N-1}^{\infty} \left(1-\frac{2\arctan\frac{\delta}{\sqrt{1-\delta^2}}}{\pi}\right)^k}
    \le {C'_6 N^{-\varepsilon}}.
\end{align*}
For sufficiently large $N$, 
\begin{align*}
    \mathbb{E}[T_{\text{tree}}]\le (1+0.1\varepsilon)\frac{\pi}{\lambda}N\log N+\sum_{k=1}^{\infty}(1+0.1k\varepsilon)\frac{\pi}{\lambda}N^{1-0.1k\varepsilon}\log N\le (1+\varepsilon)\frac{\pi}{\lambda}N\log N.
\end{align*}
{Since the total arrival rate in both the original cylinder and the rescaled cylinder is $2\pi N$, \eqref{eq:t_T and T} and $\omega_{N,\lambda}=t_{T_{\text{tree}}}$ give}
    \begin{align*}
        {\E[\omega_{N, \lambda}]=\E[t_{T_{\text{tree}}}]=\frac{\E[T_{\text{tree}}]}{2\pi N}\le (1+\varepsilon)\frac{\log N}{2\lambda}.}
    \end{align*}
    {This proves the upper bound of $\E[\omega_{N, \lambda}]$.}

\end{proof}


\section{{The expected number of trees}}\label{sec:Expected number of trees}
{This section proves Theorem \ref{thm: number of trees} by using the backward $\operatorname{CHL}_N$  process. The resulting linear recursion yields \(\E[\mathcal{N}_{\infty}]\), and the same additivity argument gives \(\E[D_\infty]=1\).}

\begin{proof}[Proof of Theorem \ref{thm: number of trees}]
    
{Recall} that $\mathcal{N}_{t}$ is the number of trees at time $t$ in {the} standard $\operatorname{CHL}_N$ process. Note that $\mathcal{N}_{t}$ increases as $t$ increases, so there exists $\mathcal{N}_{\infty}$ such that
    \begin{align*}
        \mathcal{N}_{t} \to \mathcal{N}_{\infty} \quad \mathrm{a.s.} .
    \end{align*}
By {the} Monotone Convergence Theorem, 
    \begin{align*}
        \mathbb{E}[\mathcal{N}_{t}] \to \mathbb{E}[\mathcal{N}_{\infty}].
    \end{align*}
Now we compute $\mathbb{E}[\mathcal{N}_{\infty}]$ by {a recursion for conditional expectations} and use the backward $\operatorname{CHL}_N$ process  $\tilde{\chl}^{1,\lambda/N}_t(z)$, which is equal in distribution to {the} $\text{CHL}_N$ process at any fixed time. Under $\tilde{\chl}^{1,\lambda/N}_t(z)$,  {if a slit is grown at} $x \in \mathbb T_{1}$ at some time $t$, then all trees grown on the interval $[x-2\arctan\frac{\delta}{\sqrt{1-\delta^2}}, x+2\arctan\frac{\delta}{\sqrt{1-\delta^2}}]$ will attach to the slit ({see} Figure \ref{fig:num_of_trees}). 
\begin{figure}
    \centering

\tikzset{every picture/.style={line width=0.75pt}} 

\begin{tikzpicture}[x=0.75pt,y=0.75pt,yscale=-0.7,xscale=0.7]

\draw    (32,210) -- (292.5,210) ;
\draw [color={rgb, 255:red, 208; green, 2; blue, 27 }  ,draw opacity=1 ][line width=1.5]    (141,210) -- (169,210) ;
\draw [color={rgb, 255:red, 74; green, 144; blue, 226 }  ,draw opacity=1 ]   (41,185) -- (66,210) ;
\draw [color={rgb, 255:red, 74; green, 144; blue, 226 }  ,draw opacity=1 ]   (53.5,197.5) -- (85,162) ;
\draw [color={rgb, 255:red, 74; green, 144; blue, 226 }  ,draw opacity=1 ]   (44.25,154.75) -- (69.25,179.75) ;
\draw [color={rgb, 255:red, 74; green, 144; blue, 226 }  ,draw opacity=1 ]   (57,118) -- (78,170) ;
\draw [color={rgb, 255:red, 74; green, 144; blue, 226 }  ,draw opacity=1 ]   (287,186) -- (262,211) ;
\draw [color={rgb, 255:red, 74; green, 144; blue, 226 }  ,draw opacity=1 ]   (274.5,198.5) -- (243,163) ;
\draw [color={rgb, 255:red, 74; green, 144; blue, 226 }  ,draw opacity=1 ]   (263,164) -- (279.75,192.75) ;
\draw [color={rgb, 255:red, 74; green, 144; blue, 226 }  ,draw opacity=1 ]   (289,139) -- (271.38,178.38) ;
\draw [color={rgb, 255:red, 74; green, 144; blue, 226 }  ,draw opacity=1 ]   (182,185) -- (157,210) ;
\draw [color={rgb, 255:red, 74; green, 144; blue, 226 }  ,draw opacity=1 ]   (175.19,157.69) -- (143.69,122.19) ;
\draw [color={rgb, 255:red, 74; green, 144; blue, 226 }  ,draw opacity=1 ]   (158,163) -- (174.75,191.75) ;
\draw [color={rgb, 255:red, 74; green, 144; blue, 226 }  ,draw opacity=1 ]   (184,138) -- (166.38,177.38) ;
\draw [color={rgb, 255:red, 74; green, 144; blue, 226 }  ,draw opacity=1 ]   (117.69,184) -- (142.69,209) ;
\draw [color={rgb, 255:red, 74; green, 144; blue, 226 }  ,draw opacity=1 ]   (123.16,168.19) -- (133.66,158.5) ;
\draw [color={rgb, 255:red, 74; green, 144; blue, 226 }  ,draw opacity=1 ]   (141.69,162) -- (124.94,190.75) ;
\draw [color={rgb, 255:red, 74; green, 144; blue, 226 }  ,draw opacity=1 ]   (113,160) -- (133.31,176.38) ;
\draw    (327,210) -- (587.5,210) ;
\draw [color={rgb, 255:red, 208; green, 2; blue, 27 }  ,draw opacity=1 ][line width=1.5]    (450,210) -- (450,186) ;
\draw [color={rgb, 255:red, 74; green, 144; blue, 226 }  ,draw opacity=1 ]   (333.67,182.25) -- (355.54,210.03) ;
\draw [color={rgb, 255:red, 74; green, 144; blue, 226 }  ,draw opacity=1 ]   (344.61,196.14) -- (380.08,164.62) ;
\draw [color={rgb, 255:red, 74; green, 144; blue, 226 }  ,draw opacity=1 ]   (340.48,152.6) -- (362.34,180.38) ;
\draw [color={rgb, 255:red, 74; green, 144; blue, 226 }  ,draw opacity=1 ]   (357.48,117.61) -- (372.19,171.73) ;
\draw [color={rgb, 255:red, 74; green, 144; blue, 226 }  ,draw opacity=1 ]   (584.12,176.53) -- (572.38,209.87) ;
\draw [color={rgb, 255:red, 74; green, 144; blue, 226 }  ,draw opacity=1 ]   (578.25,193.2) -- (534.5,174.8) ;
\draw [color={rgb, 255:red, 74; green, 144; blue, 226 }  ,draw opacity=1 ]   (552.97,167.06) -- (580.5,185.75) ;
\draw [color={rgb, 255:red, 74; green, 144; blue, 226 }  ,draw opacity=1 ]   (565.62,133.28) -- (566.73,176.4) ;
\draw [color={rgb, 255:red, 74; green, 144; blue, 226 }  ,draw opacity=1 ]   (484.25,187.94) -- (450.39,198.1) ;
\draw [color={rgb, 255:red, 74; green, 144; blue, 226 }  ,draw opacity=1 ]   (491.2,160.66) -- (480.29,114.47) ;
\draw [color={rgb, 255:red, 74; green, 144; blue, 226 }  ,draw opacity=1 ]   (473.55,157.19) -- (474.67,190.45) ;
\draw [color={rgb, 255:red, 74; green, 144; blue, 226 }  ,draw opacity=1 ]   (508.29,147.5) -- (474.11,173.82) ;
\draw [color={rgb, 255:red, 74; green, 144; blue, 226 }  ,draw opacity=1 ]   (419.97,188.13) -- (449.67,207.31) ;
\draw [color={rgb, 255:red, 74; green, 144; blue, 226 }  ,draw opacity=1 ]   (421.99,171.52) -- (430.21,159.84) ;
\draw [color={rgb, 255:red, 74; green, 144; blue, 226 }  ,draw opacity=1 ]   (438.8,161.57) -- (428.48,193.2) ;
\draw [color={rgb, 255:red, 74; green, 144; blue, 226 }  ,draw opacity=1 ]   (410.33,165.66) -- (433.64,177.39) ;
\draw    (294,180) -- (326,180) ;
\draw [shift={(328,180)}, rotate = 180] [color={rgb, 255:red, 0; green, 0; blue, 0 }  ][line width=0.75]    (10.93,-3.29) .. controls (6.95,-1.4) and (3.31,-0.3) .. (0,0) .. controls (3.31,0.3) and (6.95,1.4) .. (10.93,3.29)   ;

\end{tikzpicture}

    \caption{Trees after a slit is attached in the backward process $\tilde{\chl}^{1,\lambda/N}_t(z)$ \label{fig:num_of_trees}}
    \label{fig:enter-label}
\end{figure}
{Denote by $X_k$ the number of trees after $k$ particles have grown on $\mathbb T_1$, and let $B_i$ $(1\le i\le X_k)$ be the event that the $i$-th existing tree is not covered by the next backward slit. With this observable,}
\begin{align*}
    {X_{k+1}=1+\sum_{j=1}^{X_k}\mathbf{1}_{B_j}.}
\end{align*}
{Let $\mathcal G_k$ be the sigma-field generated by the current backward configuration. Conditional on $\mathcal G_k$, the new slit position is uniform on $\mathbb T_1$. For each existing tree, the set of slit positions that cover it has length $4\arctan\frac{\delta}{\sqrt{1-\delta^2}}$, independently of the tree. Therefore}
\begin{align*}
    {\mathbb{E}[X_{k+1} | \mathcal G_k]= 1+ \sum_{j=1}^{X_k}\mathbb{P}(B_j\mid \mathcal G_k)=1+\left(1-\frac{2}{\pi}\arctan\frac{\delta}{\sqrt{1-\delta^2}}\right)X_k,}
\end{align*}
{where no independence among the events $B_j$ is used. Taking expectations gives}
\begin{align}\label{recursion of number of trees}
    \mathbb{E}[X_{k+1} ] =1+\left(1-\frac{2}{\pi}\arctan\frac{\delta}{\sqrt{1-\delta^2}}\right) \mathbb{E}[X_k].
\end{align}
Note that
\begin{align*}
    \lim_{k\to \infty}\mathbb{E}[X_k]=\mathbb{E}[\mathcal{N}_{\infty}].
\end{align*}
{Combining this with \eqref{recursion of number of trees},} we have 
    \begin{align*}
        \mathbb{E}[\mathcal{N}_{\infty}]= \frac{\pi}{2 \arctan\frac{\delta}{\sqrt{1-\delta^2}}}.
    \end{align*}
    With \eqref{eq-def-delta}, 
    \begin{align*}
        \lim_{N\to \infty}\frac{\mathbb{E}[\mathcal{N}_{\infty}] }{N }=\frac{\pi}{\lambda}.
    \end{align*}
\end{proof}
{We next prove Corollary \ref{cor:degree of a single particle}, which gives the asymptotics of the average degree of a single
 particle.}

\begin{proof}[Proof of Corollary \ref{cor:degree of a single particle}]{Recall that $D_t$ is the number of particles directly attached to the chosen particle by time $t$. Since $D_t$ is nondecreasing in $t$, there exists $D_{\infty}$ such that}
    \begin{align*}
        {D_{t} \to D_{\infty} \quad \mathrm{a.s.} .}
    \end{align*}
{By the Monotone Convergence Theorem,}
    \begin{align*}
        {\mathbb{E}[D_{t}] \to \mathbb{E}[D_{\infty}].}
    \end{align*}
{Define $T(x)$ as the expected number of direct attachments to a single marked interval $I \subset \mathbb{T}_1$ when $|I|=x$. Then by \eqref{eq-[x-2arc,x+2arc]} we have $\mathbb{E}[D_{\infty}]=T(4\arctan{\tfrac{\delta}{\sqrt{1-\delta^2}}})$. Note that $T(x)$ is nonnegative and $T(0) = 0$. By Theorem \ref{thm: number of trees},} 
\begin{align*}
    T(2\pi)=\mathbb{E}[\mathcal{N}_{\infty}]=\frac{\pi}{2 \arctan\frac{\delta}{\sqrt{1-\delta^2}}}.
\end{align*}
Moreover, by additivity of expectation, dividing the interval $I'$ with length $x+y$ into two intervals with lengths $x,y$ yields $T(x+y)=T(x)+T(y)$ for $x,y \geq 0$ such that $x + y \leq 2\pi$. So $T(x)$ is linear, which implies that
\begin{align*}
    T(x)=\frac{x}{4 \arctan\frac{\delta}{\sqrt{1-\delta^2}}}.
\end{align*}
{Therefore $\mathbb{E}[D_{\infty}]=T(4\arctan{\tfrac{\delta}{\sqrt{1-\delta^2}}})=1$, which completes our proof.}
\end{proof}

\section*{Acknowledgements}
{We thank Xinyi Li for fruitful discussions and a careful reading of an earlier version of the manuscript, and thank Yuan Zhang for suggesting the problem of the tree completion time. Eviatar B. Procaccia is supported by DFG grant 5010383. Yuxuan Zong is supported by National Key R\&D Program of China (No.\ 2021YFA1002700 and No.\ 2020YFA0712900) and thanks BICMR in Peking University for a visiting position in 2024/25 during which time this work was completed.}

\appendix
\section{Chernoff Bound}
\begin{lemma}\label{lem:Chernoff Bound}
    Let \( X_1, X_2, \ldots, X_n \) be i.i.d. non-negative random variables satisfying:
\begin{align*}
    \mathbb{P}(X_1 > x) \leq e^{-\frac{x}{L}}, \quad \forall x > 0,
\end{align*}
where $L$ is a parameter. Denote $S_k=X_1+\cdots+X_k$, then
\begin{align*}
    P(S_k > x) \leq \begin{cases}
1, & x \leq k L, \\
\left( \dfrac{x e}{k L} \right)^k e^{-\frac{x}{L}}, & x > k L.
\end{cases}
\end{align*}

\end{lemma}
\begin{proof}
    The MGF of \( X_i \) for \( \lambda < \frac{1}{L} \) satisfies:
\begin{align*}
    \mathbb{E}[e^{\lambda X_i}] =1+ \int_0^\infty \lambda e^{\lambda t} P(X_i \geq t) \, dt \leq 1+\lambda \int_0^\infty e^{\lambda t} e^{-t / L} \, dt= \frac{1}{1 - \lambda L}.
\end{align*}
Since the \( X_i \) are i.i.d., the MGF of $S_k$ is:
\begin{align*}
    M_{S_k}(\lambda) = \left( \mathbb{E}[e^{\lambda X_1}] \right)^k \leq \left( \frac{1}{1 - \lambda L} \right)^k.
\end{align*}
The Chernoff bound yields:
\begin{align*}
P(S_k > x) \leq \inf_{\lambda \in \bigl(0, \frac{1}{L}\bigr)} e^{-\lambda x} M_{S_k}(\lambda) \leq \inf_{\lambda \in \bigl(0, \frac{1}{L}\bigr)} e^{-\lambda x} \left( \frac{1}{1 - \lambda L} \right)^k.
\end{align*}
Let \( f(\lambda) = -\lambda x - k \ln \left( 1 - \lambda L \right) \). Taking {the} derivative and setting to zero:
\begin{align*}
    f'(\lambda) = -x + \frac{k L}{1 - \lambda L} = 0 \implies \lambda = \frac{1}{L} - \frac{k}{x}.
\end{align*}
This requires \( x > k L \) to ensure \( \lambda \in (0, 1/L) \).
Substituting the optimal \( \lambda \):
\begin{align*}
    P(S_k > x) \leq \exp\left( -\left( \frac{1}{L} - \frac{k}{x} \right) x + k \ln \left( \frac{x}{k L} \right) \right) = \left( \frac{x e}{k L} \right)^k e^{-\frac{x}{L}}.
\end{align*}
{This completes the proof.}
\end{proof}

\bibliographystyle{alpha}
\bibliography{bibliography}

\end{document}